\documentclass[reqno]{amsart}
\usepackage{amsmath}
\usepackage{amsfonts}
\usepackage{amssymb,enumerate}
\usepackage{amsthm}
\usepackage[all]{xy}
\usepackage{rotating}
\usepackage{hyperref}
\theoremstyle{plain}
\newtheorem{lem}{Lemma}[section]
\newtheorem{cor}[lem]{Corollary}
\newtheorem{prop}[lem]{Proposition}
\newtheorem{thm}[lem]{Theorem}

\newtheorem{intthm}{Theorem}

\theoremstyle{definition}
\newtheorem{defn}[lem]{Definition}

\newtheorem{disc}[lem]{Remark}
\newtheorem{rmk}[lem]{Remark}

\newtheorem{notn}[lem]{Notation}
\newtheorem{fact}[lem]{Fact}
\newtheorem{para}[lem]{}

\newcommand{\cat}[1]{\mathcal{#1}}

\newcommand{\catd}{\cat{D}}

\newcommand{\catb}{\cat{B}}

\newcommand{\catbc}{\cat{B}_C}

\newcommand{\pd}{\operatorname{pd}}	
\newcommand{\gdim}{\mathrm{G}\text{-}\!\dim}	
\newcommand{\gkdim}[1]{\mathrm{G}_{#1}\text{-}\!\dim}	
\newcommand{\gcdim}{\gkdim{C}}	
	
\newcommand{\id}{\operatorname{id}}	
\newcommand{\fd}{\operatorname{fd}}

\newcommand{\depth}{\operatorname{depth}}

\newcommand{\edim}{\operatorname{edim}}

\newcommand{\lotimes}{\otimes^{\mathbf{L}}}
\newcommand{\HH}{\operatorname{H}}
	
\newcommand{\coker}{\operatorname{Coker}}

\newcommand{\im}{\operatorname{Im}}
\newcommand{\shift}{\mathsf{\Sigma}}

\newcommand{\Ker}{\operatorname{Ker}}

\newcommand{\ideal}[1]{\mathfrak{#1}}
\newcommand{\m}{\ideal{m}}
\newcommand{\n}{\ideal{n}}

\newcommand{\fm}{\ideal{m}}
\newcommand{\fn}{\ideal{n}}

\newcommand{\comp}[1]{\widehat{#1}}
\newcommand{\ol}{\overline}
\newcommand{\wti}{\widetilde}

\newcommand{\bbn}{\mathbb{N}}

\newcommand{\xra}{\xrightarrow}

\newcommand{\vf}{\varphi}

\newcommand{\y}{\mathbf{y}}

\newcommand{\x}{\mathbf{x}}

\newcommand{\f}{\mathbf{f}}

\renewcommand{\geq}{\geqslant}
\renewcommand{\leq}{\leqslant}

\newcommand{\Ext}[4][R]{\operatorname{Ext}_{#1}^{#2}(#3,#4)}	
\newcommand{\Rhom}[3][R]{\mathbf{R}\!\operatorname{Hom}_{#1}(#2,#3)}	
\newcommand{\Lotimes}[3][R]{#2\otimes^{\mathbf{L}}_{#1}#3}

\newcommand{\Tor}[4][R]{\operatorname{Tor}^{#1}_{#2}(#3,#4)}

\newcommand{\gc}{\text{G}_{C}}

\numberwithin{equation}{lem}

\begin{document}

\bibliographystyle{amsplain}

\title{Contracting endomorphisms and dualizing complexes}

\author[S. Nasseh]{Saeed Nasseh}
\address{Mathematics Department,
NDSU Dept 2750,
PO Box 6050,
Fargo, ND 58108-6050
USA}
\email{saeed.nasseh@ndsu.edu}
\urladdr{http://www.ndsu.edu/pubweb/\~{}nasseh/}

\author[S. Sather-Wagstaff]{Sean Sather-Wagstaff}
\email{sean.sather-wagstaff@ndsu.edu}
\urladdr{http://www.ndsu.edu/pubweb/\~{}ssatherw/}

\thanks{This material is based on work supported by North Dakota EPSCoR and 
NSF Grant EPS-0814442.
Sather-Wagstaff was supported in part by a grant from the NSA}

%\date{\today}

%\dedicatory{}


\begin{abstract}
We investigate how one can detect the dualizing property for a chain complex over a commutative
local noetherian ring $R$. Our focus is on homological properties of contracting endomorphisms of $R$,
e.g., the Frobenius endomorphism when $R$ contains a field of positive characteristic.
\end{abstract}

\maketitle

%\tableofcontents

\section*{Introduction} \label{sec0}

Throughout this paper, the term ``ring'' means ``commutative noetherian ring with identity'',
and ``module'' means ``unital module''. 
A ring is ``complete'' if it is complete (i.e., separated and complete) with respect to its Jacobson radical.
Let $R$ be a ring. For this section, assume that $(R,\m,k)$ is local.

An idea in commutative algebra that is now standard is the following: interesting properties of $R$ can be detected by
homological conditions on $k$; when $R$ contains a field of positive characteristic, such properties of $R$ can be detected
similarly by $^n\! R$. Here $^n\!R$ is the additive abelian group $R$ viewed as an $R$-module
via restriction of scalars along the $n$th iterated Frobenius map $f^n_R\colon R\to R$ given by $r\mapsto r^{p^n}$.

The somewhat canonical example of this is Auslander, Buchsbaum, Kunz, Rodicio, and Serre's 
work~\cite{auslander:hdlr,kunz:corlrocp,rodicio:oaroa,serre:sldhdaedmn} characterizing regular rings in terms of finite projective 
dimension of $k$ and finite flat dimension  of $^n\! R$. Analogous characterizations of the Gorenstein property are built from
Auslander and Bridger's G-dimension~\cite{auslander:smt} (or using similar ideas) by
Goto, Iyengar, Sather-Wagstaff, Takahashi, and Yoshino~\cite{goto:aponlrocp,iyengar:golh,takahashi:ccmlrbfm}.

A comparable characterization of the dualizing property for $R$-complexes in~terms of derived reflexive behavior of $k$
goes back to Hartshorne and Grothendieck~\cite{hartshorne:rad}.
The point of this paper is to give similar characterizations of dualizing complexes with respect to $^n\!R$.
We frame the conversation in terms of Christensen's semidualizing complexes~\cite{christensen:scatac}
(coming from Avramov and Foxby's relative dualizing complexes~\cite{avramov:rhafgd}),
and following Avramov, Iyengar, and Miller~\cite{avramov:holh} in terms of contracting endomorphisms.
(See Section~\ref{sec121004a} for terminology and background results.)
A special case of one of our main results is the following, which we prove in~\ref{para120927a}.

\begin{intthm}\label{cor120711a}
Let $\vf\colon R\to R$ be a module-finite
contracting endomorphism, and let $C$ be a semidualizing $R$-complex.
Let $^n\!R$ be the additive abelian group $R$ viewed as an $R$-module
via restriction of scalars along the $n$-fold composition $\vf^n\colon R\to R$.
Then the following conditions are equivalent:
\begin{enumerate}[\rm(i)]
\item\label{cor120711a1}
$C$ is a dualizing $R$-complex.

\item\label{cor120711a2}
$C\sim \Rhom{^{n}\!R}C$ for some $n>0$.

\item\label{cor120711a3}
$\gcdim {}^n\!R<\infty$ and $C$ is derived  $\Rhom{^{n}\!R}C$-reflexive for some $n>0$.

\item\label{cor120711a4}
$\gcdim {}^n\!R<\infty$ for infinitely many $n>0$.
\end{enumerate}
If $R$ has a dualizing complex $D$, then these conditions are equivalent to the following:
\begin{enumerate}[\rm(v)]
\item $\gcdim {}^n\!R<\infty$  and $^n\!R\otimes^{\textbf{L}}_R \Rhom CD$
is derived  $\Rhom CD$-reflexive for some $n>0$.
\end{enumerate}
\end{intthm}

A standard technique for working with the Frobenius involves reducing to the case where $R$ is $F$-finite. 
The next result shows how this works in our setting; it is contained in Theorem~\ref{thm120712a'}.

\begin{intthm}\label{prop120910a}
Let $R$ be a  local ring of prime
characteristic $p>0$, and let $C$ be a semidualizing $R$-complex.
Then the following conditions are equivalent:
\begin{enumerate}[\rm(i)]
\item\label{prop120910a1}
$C$ is a dualizing $R$-complex.
\item\label{prop120910a1'}
There is a complete weakly \'etale  $F$-finite local $R$-algebra $S$ such that
$\Lotimes{S}C$ is dualizing for $S$.
\item\label{prop120910a4'}
There is a complete weakly \'etale $F$-finite local $R$-algebra $S$ such that for infinitely many $n>0$ one has
$\gkdim{\Lotimes{S}C} f_S^n<\infty$.
\item
\label{prop120910a5'} 
There is a complete weakly \'etale $F$-finite local $R$-algebra $S$ such that
for some $n>0$ one has $\gkdim{\Lotimes{S}C} f_S^n<\infty$  and $^n\!S\otimes^{\textbf{L}}_S \Rhom[S]{\Lotimes SC}{D^S}$
is derived  $\Rhom[S]{\Lotimes SC}{D^S}$-reflexive, where $D^S$ is a dualizing $S$-complex.
\end{enumerate}
\end{intthm}

It is worth noting that one of the focuses of this paper involves developing a similar method for
reducing to the module-finite situation for other contracting endomorphisms.

We conclude this section by summarizing the contents of the paper.
Section~\ref{sec121004a} contains terminology and background content.
Section~\ref{sec120830a} consists of analyses of a construction like $\Rhom{^{n}\!R}C$
that is better suited for endomorphisms that are not module-finite.
In Section~\ref{sec120613a} we prove results including Theorem~\ref{cor120711a} above about general contracting
endomorphisms, and in
Section~\ref{sec120515a} we focus briefly on the Frobenius endomorphism.
Finally, Appendix~\ref{sec120627a} contains a somewhat general construction of module-finite contracting endomorphisms.

\section{Semidualizing complexes and $\gc$-dimension} \label{sec121004a}

In this section, we recall  definitions and background material on
semidualizing complexes and related notions. We begin by specifying our notation for 
complexes and derived categories.
The reader may find~\cite{gelfand:moha,hartshorne:rad,verdier:cd, verdier:1} to be useful for
more background.

\begin{para}
In this paper, $R$-complexes are indexed homologically
$$
M=\cdots\xra{\partial^M_{i+1}}
M_{i}\xra{\partial^M_{i}} M_{i-1}\xra{\partial^M_{i-1}}
\cdots.
$$
For each integer $i$,
the $i$th \emph{suspension} (or \emph{shift}) of
$M$, denoted $\shift^i M$, is the complex with
$(\shift^i M)_n=M_{n-i}$ and
$\partial_n^{\shift^i M}=(-1)^i\partial_{n-i}^M$.

The derived category of the category of $R$-modules is denoted $\catd(R)$.
Isomorphisms in $\mathcal{D}(R)$ are identified by the
symbol $\simeq$ and isomorphisms up to shift
are designated by $\sim$.

Fix $R$-complexes $M$ and $N$. Let
$\inf(M)$ and $\sup(M)$ denote the infimum and supremum,
respectively, of the set
$\{n\in\mathbb{Z}\mid\HH_n(M)\neq 0\}$, with the conventions $\sup(\emptyset)=-\infty$ and $\inf(\emptyset)=\infty$.
The complex $M$ is \emph{homologically bounded} if $\HH_i(M)=0$
for all $|i|\gg 0$; it is \emph{degree-wise homologically finite} if each $\HH_i(M)$ is finitely generated; and
it is \emph{homologically finite}
if $\oplus_i\HH_i(M)$  is finitely generated.
If $M$ is degree-wise homologically finite and $\inf(M)\geq -\infty$, then $M$ admits a 
\emph{degree-wise finite free resolution}, that is, an isomorphism $F\xra\simeq M$ in $\catd(R)$ such that
each $F_i$ is a finitely generated free $R$-module and $F_i=0$ for $i<\inf(M)$.

Let $M\otimes_R^{\textbf{L}} N$ and $\Rhom MN$ denote the left-derived
tensor product and right-derived homomorphism complexes.
Let $\pd_R(M)$, $\fd_R(M)$, and $\id_R(M)$ denote the projective, flat, and injective
dimensions of $M$, as in \cite{avramov:hdouc}.
A ring homomorphism $R\to S$ has \emph{finite flat dimension} if $\fd_R(S)$ is finite.
When $R$ is a local ring with residue field $k$, the \emph{depth} of
$M$ is $\depth_R(M):=-\sup(\Rhom kM)$.
\end{para}

The ideas behind semidualizing complexes go back, e.g. to
Grothendieck's dualizing complexes~\cite{hartshorne:rad} and the relative dualizing complexes of Avramov and Foxby~\cite{avramov:rhafgd}.\footnote{The 
history summarized in this section is skeletal at best. For a more thorough discussion, the interested reader may find~\cite{sather:bnsc} helpful.}
The generality that we work in for this paper is from Christensen~\cite{christensen:scatac}.

\begin{defn}\label{def semi-dual complex}
An $R$-complex $C$ is \emph{semidualizing} if
it is homologically finite
and the ``homothety morphism'' $\chi^R_C\colon R\to\Rhom CC$ is an isomorphism
in $\mathcal{D}(R)$.
An $R$-complex $D$ is \emph{dualizing} if it is semidualizing and
has finite injective dimension.
\end{defn}

\begin{fact}\label{admitting dualizing complex}
If $R$ is a homomorphic image of a Gorenstein ring, e.g., if
$R$ is complete, then $R$ admits a
dualizing complex by \cite[\S V.10]{hartshorne:rad}.
\end{fact}

\begin{fact}\label{semi-dual complex base change}
Let $\varphi\colon (R,\fm)\to (S,\fn)$ be a local ring homomorphism
of finite flat dimension, and let $M$ be a
homologically finite $R$-complex.
From \cite[(5.7) Proposition]{christensen:scatac} and \cite[Theorem 4.5]{frankild:rrhffd}
we know that $S\otimes_R^{\textbf{L}} M$ is semidualizing for $S$
if and only if $M$ is semidualizing for $R$.
When $\vf$ is flat, the complex $S\otimes_R^{\textbf{L}} M$ is dualizing for $S$
if and only if $M$ is dualizing for $R$ and $S/\m S$ is Gorenstein
by \cite[(4.2) Proposition, (5.1) Theorem]{avramov:lgh}.\footnote{This can be done more generally using Gorenstein homomorphisms, but we do not need that
level of generality here; see~\cite{avramov:lgh}.}
\end{fact}

The next categories
come from Avramov and Foxby~\cite{avramov:rhafgd} and Christensen~\cite{christensen:scatac}.

\begin{defn}\label{Auslander class}
Let $C$ be a semidualizing $R$-complex.
The \emph{Auslander class} with respect
to $C$ is the full subcategory
$\mathcal{A}_C(R)\subseteq\mathcal{D}(R)$
consisting of the homologically bounded $R$-complexes $M$ such that
$C\otimes^{\textbf{L}}_R M$ is homologically bounded
and the natural morphism $\gamma^C_M\colon  M\to
\Rhom C{C\otimes^{\textbf{L}}_R M}$
is an isomorphism in $\mathcal{D}(R)$.
The \emph{Bass class} with respect
to $C$ is the full subcategory
$\mathcal{B}_C(R)\subseteq\mathcal{D}(R)$
consisting of the homologically bounded $R$-complexes $M$ such that
$\Rhom CM$ is homologically bounded
and the natural morphism $\xi^C_M\colon  \Lotimes{C}{\Rhom CM}\to M$
is an isomorphism in $\mathcal{D}(R)$.
\end{defn}

\begin{defn}\label{C-ref complex}
Let $C$  and $M$ be $R$-complexes. We set $M^{\dag_C}:=
\Rhom MC$ and $M^{\dag_C\dag_C}:=(M^{\dag_C})^{\dag_C}$.
The $R$-complex $M$ is
\emph{derived $C$-reflexive} when the complexes $M$ and $M^{\dag_C}$ are homologically finite and
the ``biduality morphism''
$\delta^C_M\colon  M\to M^{\dag_C\dag_C}$  is an
isomorphism in $\mathcal{D}(R)$. \footnote{Avramov,  Iyengar, and Lipman~\cite[Theorem 2]{avramov:rrc1} show that this definition is redundant when $C$ is semidualizing.}
\end{defn}

\begin{defn}\label{def of gcdim for complexes}
Let $C$ be a semidualizing $R$-complex. Set
$$
\gcdim_R(M):=\begin{cases}
\inf(C)-\inf(\Rhom MC) & \text{if $M$ is derived $C$-reflexive} \\
\infty & \text{otherwise.}
\end{cases}
$$
When $C=R$ we write $\gdim_R(M)$ in place of $\gcdim_R(M)$;
this is the \emph{G-dimension} of
Auslander and Bridger~\cite{auslander:smt}
and Yassemi~\cite{yassemi:gd}.
\end{defn}

\begin{fact}[\protect{\cite[(3.14) Theorem]{christensen:scatac}}]\label{ABE formula}
Let $C$ be a semidualizing $R$-complex, and let $M$ be an $R$-complex
such that $\gcdim_R(M)<\infty$. Then
$$
\gcdim_R(M)=\depth(R)-\depth_R(M).
$$
\end{fact}

\begin{fact}\label{C-RHom(C,D)}
Assume that $R$ has a dualizing
complex $D$.  Each homologically finite $R$-complex $M$ is
derived $D$-reflexive by~\cite[Proposition V.2.1]{hartshorne:rad} or~\cite[(8.4) Proposition]{christensen:scatac}.
Furthermore, for each semidualizing $R$-complex $C$, the complex $C^{\dagger_D}$ is also semidualizing;
see~\cite[(2.11)  Theorem]{christensen:scatac}.
\end{fact}

\begin{defn}
Let $R$ be local with residue field $k$, and let $M$ be a homologically finite
$R$-complex. The \emph{Poincar\'{e}} and \emph{Bass series} of $M$ are the formal Laurant series
\begin{align*}
P_M^R(t)&:=\sum_{i\in \mathbb{Z}}\dim_k(\Tor iMk)t^i &&\text{and}
&I^M_R(t)&:=\sum_{i\in \mathbb{Z}}\dim_k(\Ext ikM)t^i.
\end{align*}
\end{defn}

\begin{fact}\label{poincare and bass}
For a semidualizing $R$-complex $C$ we have from~\cite[Theorem 4.1(a)]{foxby:ibcahtm}:
$$
I^R_R(t)=I_R^{\Rhom CC}(t)=P_C^R(t)I_R^C(t).
$$
\end{fact}

Our next topic is from Avramov, Foxby, and Herzog~\cite{avramov:solh}.

\begin{defn}\label{weakly reg}
Let $\vf\colon (R,\m)\to (S,\n)$ be a local ring homomorphism.
The \emph{semi-completion} of $\vf$ is the composition $\grave{\varphi}\colon R\to \comp S$ of 
$\vf$ and the inclusion $S\to \comp S$.

The map $\vf$ is said to be \emph{weakly regular}
if it is flat with regular closed fibre.
If $\vf$ is flat, we define the \emph{depth} and \emph{embedding dimension} of $\vf$ to be
$\depth(\vf):=\depth(S/\m S)$ and $\edim(\vf):=\edim(S/\m S)$.
If $\vf$ is weakly regular of embedding dimension 0, we say that $\vf$ is \emph{weakly \'etale}
or that $S$ is a \emph{weakly \'etale $R$-algebra}.

A \emph{regular} (resp. \emph{Gorenstein}) \emph{factorization} of $\vf$
is a diagram of local homomorphisms
$R\xra{\dot{\varphi}} R'\xra{\varphi'} S$
where $\varphi=\varphi'\dot{\varphi}$, $\dot{\varphi}$ is flat,
$R'/\fm R'$ is regular (resp. Gorenstein), and $\varphi'$ is surjective.
By \cite[(1.1) Theorem]{avramov:solh}, the semi-completion $\grave{\varphi}$
admits a regular factorization
$R\to R'\to \widehat{S}$
such that $R'$ is complete; this is called a
\emph{Cohen factorization of} $\grave{\varphi}$.

Given a regular factorization $R\xra{\dot{\varphi}} R'\xra{\varphi'}S$   for
$\varphi$, it is straightforward to show that
$\edim(\vf)\geq \edim(S/\fm S)$; this factorization is  \emph{minimal} if $\edim(\vf)= \edim(S/\fm S)$.
\end{defn}

The focus of this paper is on $\gc$-dimension of local homomorphisms, though we do require the following
slightly greater generality for a few results.
See~\cite{sather:cidfc}.

\begin{defn}\label{gcdim(f)}
Let $\varphi\colon R\to S$ be a local ring homomorphism
and $M$ a homologically finite $S$-complex.
Fix a semidualizing $R$-complex $C$ and a Cohen
factorization
$R\xra{\Dot{\varphi}} R'\xra{\varphi'}\widehat{S}$
of the semi-completion $\grave{\varphi}$.
The \emph{$\text{G}_C$-dimension of $M$
over $\varphi$} is 
$$
\gcdim_{\varphi}(M):=
\gkdim{R'\otimes_R^{\textbf{L}} C}_{R'}(\widehat{S}\otimes_S^{\textbf{L}} M)-\edim(\Dot{\varphi}).
$$
The \emph{$\text{G}_C$-dimension of $\varphi$} is
$\gcdim(\varphi):=\gcdim_{\varphi}(S)$. In the case $C=R$, we
follow \cite{iyengar:golh} and set
$\gdim_{\varphi}(M):=\gkdim{R}_{\varphi}(M)$ and $\gdim(\varphi):=\gkdim{R}(\varphi)$.
\end{defn}

\begin{fact}\label{1}
Let $\vf\colon R\to S$ be a local  homomorphism,   $C$  a semidualizing $R$-complex, and  $M$  a homologically finite $S$-complex.
\begin{enumerate}[(a)]
\item\label{item120924a}
The quantities
$\gcdim_{\varphi}(M)$, $\gcdim_{\grave{\varphi}}(\widehat{S}\otimes_S^{\textbf{L}} M)$, and
$\gkdim{\widehat{R}\otimes_R^{\textbf{L}}C}_{\widehat{\varphi}}(\widehat{S}\otimes_S^{\textbf{L}} M)$
are simultaneously finite, by an argument as in~\cite[3.4.1]{iyengar:golh}. 
\item\label{item120924b}
If $\vf$ admits a Gorenstein factorization $R\xra{\dot\vf} R'\xra{\vf'} S$, then
$\gcdim_{\vf}(M)=\gkdim{\Lotimes{R'}{C}}_{R'}(M)-\depth(\dot\vf)$, as in~\cite[3.8.~Proposition]{iyengar:golh}.
\item\label{item120924c}
If $R$ admits a dualizing complex $D$, then
$\gcdim_{\varphi}(M)<\infty$ if and only if
$M$ is in $\mathcal{A}_{C^{\dag_D}}(R)$, by~\cite[2.2.3]{sather:cidfc}.
\end{enumerate}
\end{fact}

\begin{defn}
Let $\vf\colon R\to R$
be a ring endomorphism. For $n=1,2,\ldots$ let $\vf^n$ denote the $n$-fold composition of $\vf$ with itself.
Each endomorphism $\vf^n$ defines a new $R$-module structure on $R$,
which we denote as $^{n}\!R$: specifically for $r\in R$ and $s\in {}^{n}\!R$, we have $r\cdot s =
\vf^n(r)s$.\footnote{This of course depends on $\vf$, but this notation is fairly standard.} 
If $(R,\m)$ is local, then $\vf$ is \emph{contracting} if $\vf^n(\m)\subseteq\m^2$ for $n\gg 0$.

If $R$ contains a field of  characteristic $p>0$, then the Frobenius endomorphism
$f_R\colon R\to R$ given by $r\mapsto r^p$ is a contracting endomorphism, and  $R$ is \emph{$F$-finite} when $^1R$ is finitely generated
over $R$.
\end{defn}

\section{Complexes Induced From Ring Homomorphisms} \label{sec120830a}

This section contains foundational results about the following  tool from~\cite{frankild:rrhffd} that is
central to our study of $\gc$-dimensions of local ring homomorphisms.

\begin{notn}
Let $\vf\colon R\to S$ be a local homomorphism that has a Gorenstein
factorization $R\xra{\dot{\varphi}} R'\xra{\varphi'}S$.
Given an $R$-complex $M$, we set 
$${M}(\vf):=\shift^{d}\Rhom[R']{S}{R'\otimes_R^{\textbf{L}}M}$$ 
where
$d=\depth(\dot\vf)$.
\end{notn}

\begin{disc}\label{disc120615b}
Let $\vf\colon R\to S$ be a local homomorphism that has a Gorenstein
factorization, and let $M$ be a homologically finite 
$R$-complex. The $R$-complex $M(\vf)$ is independent of the choice of Gorenstein factorization
by~\cite[Theorem 6.5(a)]{frankild:rrhffd}.
If $C$ is semidualizing for $R$, then $C(\vf)$ is semidualizing for $S$
if and only if $\gcdim(\vf)<\infty$,
by Fact~\ref{semi-dual complex base change} and~\cite[(6.1) Theorem]{christensen:scatac}.
Also, if $C$ is dualizing for $R$, then $C(\vf)$ is dualizing for $S$ by~\cite[Remark 6.7]{frankild:rrhffd}.
If $\vf$ is  module-finite, then 
$C(\vf)\simeq \Rhom{S}{C}$
by~\cite[Theorem 6.5(c)]{frankild:rrhffd}.
\end{disc}

\begin{defn}\label{defn120924a}
Let $\vf\colon R\to R$ be a contracting endomorphism.
A \emph{factorized pushout diagram} is a commutative diagram
of local ring homomorphisms
\begin{equation}
\label{eq120823a}
\begin{split}
\xymatrix{
&R\ar[r]^{\alpha}\ar[ld]_{\dot{\varphi}}\ar[dd]^<<<<<{\vf}& S\ar[dd]_<<<<<{\psi}\ar[rd]^{\dot{\psi}}&&\\
R'\ar[rd]^{\varphi'}\ar[rrr]^-{\alpha'}&&&S'\ar[ld]_{\psi'}&\\
&\wti R\ar[r]^{\wti\alpha}&\wti S.
}
\end{split}
\end{equation}
such that the maps $\alpha$ and $\alpha'$ have finite flat dimension,
the diagrams $R\xra{\dot\vf}R'\xra{\vf'}\wti R$ and $S\xra{\dot\psi}S'\xra{\psi'}\wti S$
are Gorenstein factorizations of $\vf$ and $\psi$, respectively,
and the natural morphism $\Lotimes[R']{S'}{\wti R}\to \wti S$ is an isomorphism.
\end{defn}

\begin{disc}\label{disc120615a}
Factorized pushout diagrams
exist in at least two important cases:

(1) 
Consider a commutative diagram of local ring homomorphisms
$$
\xymatrix{
(R,\fm)\ar[r]^{\alpha}\ar[d]_{\vf}&(S,\fn)\ar[d]^{\psi}\\
(\widetilde{R},\widetilde{\fm})\ar[r]^{\widetilde{\alpha}}&(\widetilde{S},\widetilde{\fn})
}
$$
such that $\alpha$ and $\wti\alpha$ are  weakly regular,
$\widetilde{S}$ is complete, and
the induced map $R/\fm\to \widetilde{S}/\widetilde{\fn}$ is separable.
Assume that $\vf$ has a minimal regular factorization
$R\xra{\dot{\varphi}} R'\xra{\varphi'} S$,
and fix a minimal Cohen factorization
$\widetilde{R}\xra{\dot{\wti\vf}} S'\xra{\wti\vf'} \widetilde{S}$
of $\widetilde{\varphi}$. 
Then~\cite[Proposition 3.2]{nasseh:cfwfa} provides a weakly regular local ring
homomorphism $\alpha'\colon R'\to S'$ satisfying the hypotheses of Lemma~\ref{lem120517a}.

(2)
For the Frobenius endomorphism, such diagrams are built in Lemma~\ref{Factorization of Frobenius}.
\end{disc}

The next few results explain behavior of the $M(\vf)$-construction with respect to factorized pushout diagrams.

\begin{lem}\label{lem120517a}
Let $C$ be a semidualizing $R$-complex,
and fix a factorized pushout diagram~\eqref{eq120823a}.
\begin{enumerate}[\rm(a)]
\item \label{lem120517a1}
The following conditions are equivalent:
\begin{enumerate}[\rm(i)]
\item $\gcdim \vf<\infty$,
\item $\gkdim{S\otimes_R^{\textbf{L}}C}\psi<\infty$,
\item $C(\vf)$ is semidualizing for $\wti R$, and
\item $(\Lotimes SC)(\psi)$ is semidualizing for $\wti S$.
\end{enumerate}
\item \label{lem120517a2}
$\gcdim \vf+\edim(\dot\vf)
=\gkdim{S\otimes_R^{\textbf{L}}C}\psi+\edim(\dot\psi)$.
\end{enumerate}
\end{lem}

\begin{proof}
From~\cite[Theorem 4.8]{frankild:rrhffd}  we have
\begin{align*}
\gkdim{R'\otimes_R^{\textbf{L}} C}_{R'}(\wti R)
&=\gkdim{S'\otimes_{R'}^{\textbf{L}}(R'\otimes_R^{\textbf{L}}C)}_{S'}(S'\otimes_{R'} \wti R)\\
&=\gkdim{S'\otimes^{\textbf{L}}_{S}(S\otimes_R^{\textbf{L}} C)}_{S'}(\wti S).
\end{align*}
This explains part~\eqref{lem120517a2} and the equivalence
(i)$\iff$(ii) from part~\eqref{lem120517a1}.
Since $\Lotimes SC$ is semidualizing for $S$ by Fact~\ref{semi-dual complex base change}, the equivalences
(i)$\iff$(iii) and (ii)$\iff$(iv)  from part~\eqref{lem120517a1} are by
Remark~\ref{disc120615b}.
\end{proof}

\begin{lem}\label{lem120615a}
Consider a factorized pushout diagram~\eqref{eq120823a}.
Given a homologically finite $R$-complex $M$,
there is an isomorphism 
$(S\lotimes_RM)(\psi)\simeq \shift^{d}\wti S\lotimes_{\wti R}M(\vf)$
in $\catd(S)$ where $d=\depth(\dot\psi)-\depth(\dot\vf)$.
\end{lem}

\begin{proof}
In the following sequence of isomorphisms in $\catd(S)$, the first and fourth steps are by definition,
and the second one is from the assumptions on  diagram~\eqref{eq120823a}:
\begin{align*}
(S\lotimes_RM)(\psi)
&\simeq\shift^{\depth(S'/\n S')}\Rhom[S']{\wti S}{S'\lotimes_S(S\lotimes_RM)}\\
&\simeq\shift^{\depth(S'/\n S')}\Rhom[S']{\Lotimes[R']{S'}{\wti R}}{S'\lotimes_{R'}(R'\lotimes_RM)}\\
&\simeq\shift^{\depth(S'/\n S')}S'\lotimes_{R'}\Rhom[R']{\wti R}{R'\lotimes_RM}\\
&\simeq \shift^{d}S'\lotimes_{R'}M(\vf)\\
&\simeq\shift^{d}(S'\lotimes_{R'} \wti R)\lotimes_{\wti R}M(\vf)\\
&\simeq \shift^{d}\wti S\lotimes_{\wti R}M(\vf).
\end{align*}
The third isomorphism follows from~\cite[4.4 Lemma]{avramov:hdouc}, and the others are routine.
\end{proof}

\begin{lem}\label{lem120615az}
Consider a factorized pushout diagram~\eqref{eq120823a}
such that $\wti R=R$ and $\wti S=S$ and $\wti \alpha=\alpha$, 
that is, such that $\vf$ and $\psi$
are endomorphisms.
Let $C$ be a semidualizing $R$-complex.
Then $C\sim C(\vf)$ as $R$-complexes
if and only if $S\lotimes_RC\sim (S\lotimes_RC)(\psi)$ as $S$-complexes.
\end{lem}

\begin{proof}
Since $C$ is semidualizing over $R$ and $\fd(\alpha)<\infty$, we know that 
$\Lotimes SC$ is semidualizing for $S$ by Fact~\ref{semi-dual complex base change}.

For the forward implication, assume that $C\sim C(\vf)$ as $R$-complexes.
This implies
$$S\lotimes_RC\sim S\lotimes_R(C(\vf))\sim (S\lotimes_RC)(\psi)$$
by Lemma~\ref{lem120615a}.

For the converse, assume that $S\lotimes_RC\sim (S\lotimes_RC)(\psi)$ as $S$-complexes.
In particular, this implies that $(S\lotimes_RC)(\psi)$ is semidualizing for $S$,
so Lemma~\ref{lem120517a}\eqref{lem120517a1} implies that $C(\vf)$ is semidualizing over $R$. 
Now, since 
$$S\lotimes_RC\sim (S\lotimes_RC)(\psi)\sim S\lotimes_R(C(\vf))$$
we conclude $C\sim C(\vf)$ by~\cite[Theorem 4.9]{frankild:rrhffd}.
\end{proof}

\begin{lem}\label{lem120725a}
Fix a factorized pushout diagram~\eqref{eq120823a}
such that $\wti R=R$ and $\wti S=S$ and $\wti \alpha=\alpha$, 
that is, such that $\vf$ and $\psi$
are endomorphisms.
Let $C$ be a semidualizing $R$-complex.
Assume that $\gcdim(\vf)<\infty$. Then
$C$ is derived $C(\vf)$-reflexive over $R$
if and only if $S\lotimes_RC$ is derived $(S\lotimes_RC)(\psi)$-reflexive over $S$.
\end{lem}

\begin{proof}
Lemma~\ref{lem120517a}\eqref{lem120517a1} implies that
$C(\vf)$ is semidualizing for $R$, and
$(\Lotimes SC)(\psi)$ is semidualizing for $S$.
Thus, the equivalence follows from
Lemma~\ref{lem120615a} and~\cite[Theorem 4.8]{frankild:rrhffd} as in the
proof of Lemma~\ref{lem120615az}.
\end{proof}

The next two results document behavior of the $M(\vf)$-construction with respect to semi-completions and compositions.

\begin{lem}\label{lem120620c}
Let $R\xra{\vf} S$ be a local homomorphism
that admits a Gorenstein
factorization, and consider the semi-completion $R\xra{\grave\vf}\comp S$.
Given a  homologically finite $R$-complex $M$,
there is an isomorphism $M(\grave\vf)\simeq \Lotimes[S]{\comp S}{M(\vf)}$ in $\catd(\comp S)$.
\end{lem}

\begin{proof}
Let $R\xra{\dot\vf}R'\xra{\vf'}S$ be a Gorenstein factorization of $\vf$.
Consider the following commutative diagram where the 
new maps are the natural ones: 
$$\xymatrix{
&R'\ar[r]^-{\gamma}\ar[rd]^-{\vf'}
&\comp{R'}\ar[rd]^-{\comp{\vf'}} \\
R\ar[ur]^-{\dot\vf}\ar[rr]^-{\vf}
&&S\ar[r]^-{\delta}&\comp S.
}$$
Since $\vf'$ is surjective, we have $\comp S\simeq\Lotimes[R']{\comp{R'}}{S}$,
and the diagram $R\xra{\gamma\dot\vf}\comp{R'}\xra{\comp{\vf'}}\comp S$
is a Gorenstein factorization of $\grave\vf=\delta\vf$.
Also, we have $d:=\depth(\dot\vf)=\depth(\gamma\dot\vf)$.
These explain the first, second, and last steps in the next sequence:
\begin{align*}
M(\grave\vf)
&\simeq \shift^d\Rhom[\comp{R'}]{\comp S}{\Lotimes{\comp{R'}}{M}}
\\
&\simeq \shift^d\Rhom[\comp{R'}]{\Lotimes[R']{\comp{R'}}{S}}{\Lotimes[R']{\comp{R'}}{(\Lotimes{R'}M})}
\\
&\simeq \shift^d\Lotimes[R']{\comp{R'}}{\Rhom[R']{S}{\Lotimes{R'}M}}
\\
&\simeq \Lotimes[R']{\comp{R'}}{\shift^d\Rhom[R']{S}{\Lotimes{R'}M}}
\\
&\simeq \Lotimes[S]{(\Lotimes[R']{\comp{R'}}{S})}{\shift^d\Rhom[R']{S}{\Lotimes{R'}M}}
\\
&\simeq \Lotimes[S]{\comp S}{M(\vf)}.
\end{align*}
The third step is by flat base change since $S$ is finite over $R'$.
The remaining steps are routine.
\end{proof}

\begin{lem}\label{lem120620a}
Let $(R,\m)\xra{\vf} (S,\n)\xra{\psi} T$ be local homomorphisms 
such that $\vf$, $\psi$, and $\psi\vf$ admit Gorenstein
factorizations.
Given a  homologically finite $R$-complex $M$,
there is an isomorphism $M(\psi\vf)\simeq M(\vf)(\psi)$ in $\catd(T)$.
\end{lem}

\begin{proof}
Case 1: $T$ is complete. 
Let $R\xra{\dot\vf} R'\xra{\vf'}S$ be a Gorenstein factorization of $\vf$,
and let $S\xra{\dot\psi} S'\xra{\psi'}T$ be a Cohen factorization of $\psi$.
Then $S'$ is complete, so the map $R'\xra{\dot\psi\vf'}S'$ has a 
minimal Cohen factorization
$R'\xra{\dot\rho}R''\xra{\rho'}S'$. 
Note that it follows from the proof of~\cite[(1.6) Theorem]{avramov:solh} that
$S'\simeq\Lotimes[R']{R''}{S}$.
From this we conclude that $\dot\psi$ and $\dot\rho$ have isomorphic
closed fibres. In particular, we have $\depth(\dot\rho)=\depth(\dot\psi)$.

Set $d'=\depth(\dot\psi)$ and $d''=\depth(\dot\vf)$.
We claim that the composition
$\dot\rho\dot\vf$ is Gorenstein and flat,
and that
$$d:=\depth(\dot\rho\dot\vf)
=\depth(\dot\rho)+\depth(\dot\vf)
=\depth(\dot\psi)+\depth(\dot\vf)=d'+d''.$$
Indeed, the composition of flat local homomorphisms is flat and local.
Furthermore, the induced map $R'/\m R'\xra{\ol{\dot\rho}} R''/\m R''$
is flat and local with closed fibre $R''/\m' R''$ where $\m'$ is the maximal ideal of $R'$.
Since $R'/\m R'$ and $R''/\m' R''$ are both Gorenstein by assumption,
the fact that $\ol{\dot\rho}$ is flat and local implies that $R''/\m R''$ is Gorenstein,
so $\dot\rho\dot\vf$ is Gorenstein. Furthermore, the fact that $\ol{\dot\rho}$ is flat and local explains the
second equality in the next sequence
\begin{align*}
\depth(\dot\rho\dot\vf)
&=\depth(R''/\m R'')\\
&=\depth(R'/\m R')+\depth(R''/\m' R'')\\
&=\depth(\dot\rho)+\depth(\dot\vf) \\
&=\depth(\dot\psi)+\depth(\dot\vf).
\end{align*}
The last step follows from the fact that $\dot\rho$ and $\dot\psi$ are both flat and have isomorphic closed fibres. 
This establishes the claim.

Thus, the diagram $R\xra{\dot\rho\dot\vf}R''\xra{\psi'\rho'}T$
is a Gorenstein factorization of $\psi\vf$:
$$\xymatrix{
&&R''\ar[rd]^-{\rho'} \\
&R'\ar[ur]^-{\dot\rho}\ar[rd]^-{\vf'}
&&S'\ar[rd]^-{\psi'} \\
R\ar[ur]^-{\dot\vf}\ar[rr]^-{\vf}
&&S\ar[ur]^-{\dot\psi}\ar[rr]^-{\psi}&&T.
}$$
This explains the first, third, sixth, and eighth steps in the next display:
\begin{align*}
M(\vf)(\psi)
&=\shift^{d''}\Rhom[S']{T}{\Lotimes[S]{S'}{\shift^{d'}\Rhom[R']{S}{\Lotimes{R'}{M}}}}\\
&\simeq\shift^{d'+d''}\Rhom[S']{T}{\Lotimes[S]{S'}{\Rhom[R']{S}{\Lotimes{R'}{M}}}}\\
&\simeq\shift^{d}\Rhom[S']{T}{\Lotimes[S]{(\Lotimes[R']{R''}{S})}{\Rhom[R']{S}{\Lotimes{R'}{M}}}}\\
&\simeq\shift^{d}\Rhom[S']{T}{\Lotimes[R']{R''}{\Rhom[R']{S}{\Lotimes{R'}{M}}}}\\
&\simeq\shift^{d}\Rhom[S']{T}{\Rhom[R'']{\Lotimes[R']{R''}{S}}{\Lotimes[R']{R''}{(\Lotimes{R'}{M})}}}\\
&\simeq\shift^{d}\Rhom[S']{T}{\Rhom[R'']{S'}{\Lotimes{R''}{M}}}\\
&\simeq\shift^{d}\Rhom[R'']{T}{\Lotimes{R''}{M}}\\
&=M(\psi\vf).
\end{align*}
The seventh step is Hom-tensor adjointness, and the others are routine.

Case 2: the general case.
Let $\grave\psi\colon S\to\comp T$ be the semi-completion of $\psi$.
Note that $\grave\psi\vf\colon R\to\comp T$ is the semi-completion of $\psi\vf$.
Thus, Lemma~\ref{lem120620c} explains the first and third isomorphisms
in the next sequence
\begin{align*}
\Lotimes[T]{\comp T}{M(\psi\vf)}
&\simeq M(\grave\psi\vf)
\simeq M(\vf)(\grave\psi)
\simeq \Lotimes[T]{\comp T}{(M(\vf)(\psi))}
\end{align*}
and the second isomorphism is from Case 1 since $\comp T$ is complete.
Hence, the conclusion $M(\psi\vf)\simeq M(\vf)(\psi)$
follows from~\cite[Lemma 1.10]{frankild:rrhffd}.
\end{proof}

The interested reader may want to compare our next two results to~\cite[Proposition 6.10]{frankild:rrhffd} which assumes that $\fd(\vf)$ is finite.

\begin{prop}\label{lem120711a}
Let $R\xra{\vf} S$ be a local homomorphism
that admits a Gorenstein
factorization, and let $C$ be a semidualizing $R$-complex.
\begin{enumerate}[\rm(a)]
\item\label{lem120711a1} Then one has  $I_S^{C(\vf)}(t)=I_{R}^{C}(t)$.
\item\label{lem120711a2}
If $\gcdim(\vf)$ is finite, then $P^S_{C(\vf)}(t)=P^{R}_{C}(t)I^S_S(t)/I^R_R(t)$.
\end{enumerate}
\end{prop}

\begin{proof}
\eqref{lem120711a1}
In the following display, the first equality is by definition:
\begin{align*}
I_S^{C(\vf)}(t)
&=I_{S}^{\shift^d\Rhom[R']{S}{\Lotimes{R'}{C}}}(t)\\
&=t^{-d}I_{S}^{\Rhom[R']{S}{\Lotimes{R'}{C}}}(t)\\
&=t^{-d}I_{R'}^{\Lotimes{R'}{C}}(t) \\
&=t^{-d}I^{R'}_{R'}(t)/P^{R'}_{\Lotimes{R'}{C}}(t) \\
&=t^{-d}I^R_R(t)I^{R'/\m R'}_{R'/\m R'}(t)/P^{R}_{C}(t) \\
&=t^{-d}I^R_R(t)t^d/P^{R}_{C}(t) \\
&=I^R_R(t)/P^{R}_{C}(t) \\
&=I_{R}^{C}(t).
\end{align*}
The third equality is from~\cite[(1.7.8) Lemma]{christensen:scatac}.
The fourth and eighth equalities are by Fact~\ref{poincare and bass}.
The fifth equality is from~\cite[Theorem]{foxby:mirufbc}.
The sixth equality is from the fact that $\dot\vf$ is Gorenstein of depth $d$,
and the remaining equalities are routine.

\eqref{lem120711a2}
Assume that $\gcdim(\vf)$ is finite, that is, that $C(\vf)$ is a semidualizing $S$-complex;
see Remark~\ref{disc120615b}.
Thus, Fact~\ref{poincare and bass} explains the first and third equalities in the next display:
\begin{align*}
I^R_R(t)P^S_{C(\vf)}(t)
&=I^R_R(t)I^S_S(t)/I_S^{C(\vf)}(t) \\
&=I^R_R(t)I^S_S(t)/I_{R}^{C}(t) \\
&=I^R_R(t)I^S_S(t)/[I_{R}^{R}(t)/P^R_C(t)] \\
&=P^{R}_{C}(t)I^S_S(t).
\end{align*}
The second equality is from part~\eqref{lem120711a1}, and the fourth equality is routine.
\end{proof}

\begin{cor}\label{lem120711b}
Let $R\xra{\vf} R$ be a local endomorphism.
Assume that $n$ is a positive integer such 
that $\vf^n$ admits a Gorenstein
factorization, and let $C$ be a semidualizing $R$-complex.
\begin{enumerate}[\rm(a)]
\item\label{lem120711b1} Then one has  $I_R^{C(\vf^n)}(t)=I_{R}^{C}(t)$.
\item\label{lem120711b2}
If $\gcdim(\vf^n)$ is finite, then $P^R_{C(\vf^n)}(t)=P^{R}_{C}(t)$.
\end{enumerate}
\end{cor}

\begin{proof}
This follows directly from Proposition~\ref{lem120711a} since $S=R$ in this case.
\end{proof}

\section{Results about Contracting Endomorphisms}\label{sec120613a}

This section contains the proof of Theorem~\ref{cor120711a} from the introduction
and other similar results for arbitrary contracting endomorphisms.
We begin with a version of~\cite[7.3.~Corollary]{iyengar:golh}
for our situation.

\begin{prop}\label{gcdimf=gcdimM}
Let $\varphi\colon (R,\m)\to S$ be a local homomorphism
and $M$ a complex of $S$-modules that is homologically finite over $R$.
Let $C$ be a semidualizing $R$-complex.
Then $\gcdim_{\varphi}(M)=\gcdim_R(M)$. In particular, the quantities $\gcdim_{\varphi}(M)$ and $\gcdim_R(M)$
are simultaneously finite.
\end{prop}

\begin{proof}
Let $\widetilde{S}$ be the $\fm$-adic completion of $S$,
and let $\widetilde{\varphi}\colon \widehat{R}\to \widetilde{S}$ be the induced map.
Let $\widehat{\varphi}\colon \widehat{R}\to \widehat{S}$ denote the map induced
on completions, and set $\comp C=\Lotimes{\comp R}{C}$. 
Consider the Koszul complex $K=K^R(\m)$ on a minimal generating sequence for $\m$. Arguing as in
the proof of~\cite[5.1.~Theorem]{iyengar:golh}, and using~\cite[Theorem 4.4]{frankild:rrhffd},
we can replace $M$ with $\Lotimes KM$ to assume that $\m$ annihilates the homology of $M$.
(The results~\cite[Proposition 4.1(a)]{sather:cidfc} and~\cite[Proposition 4.1(i)]{christensen:apac} may be helpful here.)
As in the proof of~\cite[7.1.~Theorem]{iyengar:golh}, it follows that
$\widetilde{S}\otimes_S^{\textbf{L}}M\simeq \widehat{R}\otimes_R^{\textbf{L}}M$ in $\catd(\comp R)$. 

Since the completion of $\widetilde{\varphi}$ at the
maximal ideal of $\widetilde{S}$ is $\widehat{\varphi}$,  Fact~\ref{1}  implies that
$\gcdim_{\varphi}(M)<\infty$ if and only if
$\gkdim{\widehat{C}}_{\widetilde{\varphi}}(\widetilde{S}\otimes_S^{\textbf{L}}M)<\infty$
if and only if
$\widetilde{S}\otimes_S^{\textbf{L}}M \in \mathcal{A}_{\widehat{C}^{\dagger_D}}(\widehat{R})$,
where $D$ is a dualizing complex for $\widehat{R}$.
Using the isomorphism
$\widetilde{S}\otimes_S^{\textbf{L}}M\simeq \widehat{R}\otimes_R^{\textbf{L}}M$, we conclude that
$\gcdim_{\varphi}(M)<\infty$ if and only if
$\widehat{R}\otimes_R^{\textbf{L}}M\in \mathcal{A}_{\widehat{C}^{\dagger_D}}(\widehat{R})$, that is,
if and only if $\gkdim{\widehat{C}}_{\widehat{R}}(\widehat{R}\otimes_R^{\textbf{L}}M)<\infty$,
by~\cite[(4.7) Theorem]{christensen:scatac}.
Because of the equality $\gkdim{\widehat{C}}_{\widehat{R}}(\widehat{R}\otimes_R^{\textbf{L}}M)
=\gcdim_R(M)$ from~\cite[(5.11) Corollary]{christensen:scatac}, it follows that $\gcdim_{\varphi}(M)<\infty$ if and only if
$\gcdim_R(M)<\infty$. 

For the rest of the proof, assume that $\gcdim_{\varphi}(M)$ and
$\gcdim_R(M)$ are finite. As in the proof of~\cite[3.5.~Theorem]{iyengar:golh}, using
Fact~\ref{ABE formula},
we have the first equality in the following display:
\begin{align*}
\gcdim_{\varphi}(M)
& = \depth(R) - \depth_{S}(M) \\
& = \depth(R) - \depth_{R}(M) \\
& = \gcdim_R(M).
\end{align*}
The other equalities are from~\cite[2.8.~Lemma]{iyengar:golh} and Fact~\ref{ABE formula}.
\end{proof}

The next result implies Theorem~\ref{cor120711a} from the introduction; see~\ref{para120927a}.

\begin{thm}\label{thm120615a}
Let $\vf\colon R\to R$ be a 
contracting endomorphism, and let $C$ be a semidualizing $R$-complex.
Assume that $\vf^n$ has a Gorenstein factorization for each $n\geq 1$,
e.g., this holds when $\vf$ is module-finite or $R$ is complete.
Then the following are equivalent:
\begin{enumerate}[\rm(i)]
\item\label{thm120615a1}
$C$ is a dualizing $R$-complex.

\item\label{thm120615a2}
$C\sim C(\vf^n)$ for some $n>0$.

\item\label{thm120615a3}
$\gcdim\vf^n<\infty$ and $C$ is  derived $C(\vf^n)$-reflexive for some $n>0$.

\item\label{thm120615a4}
$\gcdim\vf^n<\infty$ for infinitely many $n>0$.
\end{enumerate}
If $R$ has a dualizing complex $D$, then these conditions are equivalent to the following:
\begin{enumerate}[\rm(v)]
\item\label{thm120615a5} $\gcdim \vf^n<\infty$  and $^n\!R\otimes^{\textbf{L}}_R C^{\dagger_D}$
is derived  $C^{\dagger_D}$-reflexive for some $n>0$.
\end{enumerate}
\end{thm}

\begin{proof}
$\eqref{thm120615a1}\implies\eqref{thm120615a2}$
Assume that
$C$ is a dualizing $R$-complex.
By Remark~\ref{disc120615b}, the complex $C(\vf^n)$
is  dualizing  for $R$.
Since dualizing complexes are
unique up to shift in $\mathcal{D}(R)$,
we have $C\sim C(\vf^n)$.

$\eqref{thm120615a2}\implies\eqref{thm120615a3}$
Assume that $C\sim C(\vf^n)$ for some $n>0$.
Since $C$ is semidualizing $R$-complex,
the condition $C\sim C(\vf^n)$ implies that  $C(\vf^n)$ is semidualizing $R$-complex.
Remark~\ref{disc120615b} implies that $\gcdim\vf^n<\infty$.
Since $C$ is derived $C$-reflexive, the condition $C\sim C(\vf^n)$ implies that
$C$ is derived $C(\vf^n)$-reflexive.

$\eqref{thm120615a3}\implies\eqref{thm120615a4}$
Assume that $\gcdim\vf^n<\infty$ and $C$ is derived  $C(\vf^n)$-reflexive for some $n>0$.
Remark~\ref{disc120615b} implies that $C(\vf^n)$ is semidualizing,
and Corollary~\ref{lem120711b}\eqref{lem120711b2} implies that $C(\vf^n)$ 
has  the same Poincar\'e series as $C$.
Thus, we have $C(\vf^n)\sim C$ by
the proof of~\cite[Fact 2.28]{sather:bnsc}.
 
Thus, Lemma~\ref{lem120620a} implies that
\begin{align*}
C(\vf^{2n})
&\simeq (C(\vf^n))(\vf^n)
\sim C(\vf^n)
\sim C.
\end{align*}
Thus, we have $\gcdim_R(\vf^{2n})<\infty$ by Remark~\ref{disc120615b}.
Inductively, one shows that $\gcdim_R(\vf^{mn})<\infty$ for all $m\in\bbn$,
hence condition \eqref{thm120615a4} follows.

$\eqref{thm120615a4}\implies\eqref{thm120615a1}$
Assume that $\gcdim\vf^n<\infty$ for infinitely many $n>0$.
Fact~\ref{1} implies that
$\gcdim\vf^n<\infty$ if and only if
$\gkdim{\widehat{R}\otimes_R^{\textbf{L}}C}
\comp\vf^n<\infty$. 
Also, we know that $C$ is dualizing for $R$ if and only if 
$\Lotimes{\comp R}{C}$ is dualizing for $\comp R$.
Furthermore, $\comp\vf^n$ has a Cohen factorization for each $n$ since $\comp R$ is complete.
Thus, by passing to $\widehat{R}$,
one may assume that $R$ is complete.  Hence, $R$ has a
dualizing complex $D$ by Fact~\ref{admitting dualizing complex}.
Note that by~\cite[(2.12) Corollary]{christensen:scatac}, the $R$-complex $C^{\dag_D}$ is 
semidualizing.

By our hypothesis, $\gcdim \vf^n<\infty$ for  infinitely many $n$.
Thus by Fact~\ref{1}\eqref{item120924c}, we have
$^{n}\!R\in \mathcal{A}_{C^{\dag_D}}(R)$ and hence
$^{n}\!R\otimes_R^{\textbf{L}}C^{\dag_D}$ is homologically bounded
for infinitely many $n$.
Therefore, for infinitely many $n$ we have $\Tor i{^{n}\!R}{C^{\dag_D}}=0$ 
for all $i\gg 0$. Now \cite[6.4.~Proposition]{iyengar:golh} implies
$\pd_R(C^{\dag_D})<\infty$ and this is equivalent to $\id_R(C)<\infty$.
Thus $C$ is a dualizing complex for $R$.

To complete the proof, we assume that $R$ has a dualizing complex $D$
and prove  \eqref{thm120615a3}$\iff$(v).
To this end, we  assume that $n$ is a positive integer such that
$\gcdim\vf^n<\infty$, and we prove that $C$ is derived  $C(\vf^n)$-reflexive
if and only if $^n\!R\otimes^{\textbf{L}}_R C^{\dag_D}$
is derived  $C^{\dag_D}$-reflexive. 
Let $R\to R'\to R$ be a Gorenstein factorization of $\vf^n$.

We use the following fact from~\cite[Proposition 3.9]{frankild:rrhffd}: If $A$ and $B$ 
are semidualizing $R$-complexes,
then $A$ is derived $B$-reflexive if and only if $B^{\dagger_D}$ is derived $A^{\dagger_D}$-reflexive.
Thus, to complete the proof, we need only show that
$(\Lotimes{^{n}\!R}{C^{\dagger_D}})^{\dagger_D}\sim C(\vf^n)$. To this end, the first step in the next sequence is from
Remark~\ref{disc120615b}:
\begin{align*}
(\Lotimes{^{n}\!R}{C^{\dagger_D}})^{\dagger_D}
&\sim\Rhom[^{n}\!R]{\Lotimes{^{n}\!R}{C^{\dagger_D}}}{D(\vf^n)}\\
&\simeq\Rhom[R]{C^{\dagger_D}}{D(\vf^n)}\\
&\sim\Rhom[R]{C^{\dagger_D}}{\Rhom[R']{^{n}\!R}{\Lotimes{R'}{D}}}\\
&\simeq\Rhom[R']{\Lotimes{^{n}\!R}{C^{\dagger_D}}}{\Lotimes{R'}{D}}\\
&\simeq\Rhom[R']{\Lotimes[R']{^{n}\!R}{(\Lotimes{R'}{C^{\dagger_D}})}}{\Lotimes{R'}{D}}\\
&\simeq\Rhom[R']{^{n}\!R}{\Rhom[R']{\Lotimes{R'}{C^{\dagger_D}}}{\Lotimes{R'}{D}}}\\
&\simeq\Rhom[R']{^{n}\!R}{\Lotimes{R'}{\Rhom[R]{C^{\dagger_D}}{D}}}\\
&\simeq\Rhom[R']{^{n}\!R}{\Lotimes{R'}{C}}\\
&\sim C(\vf^n).
\end{align*}
The second, fourth, and sixth steps are from Hom-tensor adjointness.
The seventh step is by flat base change.
The eighth step is by Fact~\ref{C-RHom(C,D)}, and the other steps are routine.
\end{proof}

\begin{rmk}\label{rmk120810a}
In Theorem~\ref{thm120615a} (and its successors) we have more equivalent conditions,
but they become tedious to write down. For instance, the given conditions
are equivalent to the following:
\begin{enumerate}[\rm(i)]
\item[($\overline{\text{ii}}$)]
$C\sim C(\vf^n)$ for all $n>0$.
\end{enumerate}
Indeed, this condition clearly implies condition~\eqref{thm120712a2} from Theorem~\ref{thm120615a}.
And the proof of Theorem~\ref{thm120615a} shows that~\eqref{thm120712a1}$\implies$($\overline{\text{ii}}$).
One verifies similarly that the other conditions in Theorem~\ref{thm120615a} can be replaced with
``for all'' versions as well.
\end{rmk}

\begin{para}[Proof of Thoerem~\ref{cor120711a}] \label{para120927a}
Use Proposition~\ref{gcdimf=gcdimM}, Remark~\ref{disc120615b} and Theorem~\ref{thm120615a}.\qed
\end{para}

To state and prove results that allow us, for instance, to pass  to the completion, we introduce and briefly study
the following class of diagrams.

\begin{defn}\label{defn120712a}
Let $\vf\colon R\to R$ be a contracting endomorphism.
A commutative diagram
of local ring homomorphisms
\begin{equation}
\label{eq120712a}
\begin{split}
\xymatrix{
(R,\fm)\ar[r]^{\alpha}\ar[d]_{\vf}&(S,\fn)\ar[d]^{\psi}\\
({R},{\fm})\ar[r]^{{\alpha}}&({S},{\fn})
}
\end{split}
\end{equation}
is \emph{cows} if $S$ is \underline{co}mplete,
the map $\alpha$ is \underline{w}eakly regular,
and the map $R/\m\to S/\n$ induced by $\alpha\vf$ is \underline{s}eparable.
\end{defn}

\begin{rmk}\label{rmk120803a}
Let $\vf\colon R\to R$ be a contracting endomorphism.
One always has a trivial cows diagram~\eqref{eq120712a}: use the natural map
$\alpha\colon R\to \comp R$ and $\psi=\comp\vf$.
More interestingly, Proposition~\ref{prop120629a} shows that if 
the $R/\vf(\m)R$ is artinian and induced map $\ol\vf\colon k\to k$ is separable,
then there is a cows diagram~\eqref{eq120712a} such that
$\psi$ is  module-finite. (See also Lemma~\ref{lem120629b}.) Thus, conditions (ii') and (iii') in Theorem~\ref{thm120712a} say that
questions about $G_C$-dimensions (in the separable case) can be reduced to the
module-finite case, like reducing a Frobenius question to the $F$-finite case;
c.f. Theorem~\ref{thm120712a'}.

From another perspective, one reason to study cows diagrams is found in their similarity to Cohen factorizations:
when the map $\psi$  is module-finite, it
detects properties of $\vf$ like the surjective part $\vf'$ of a Cohen factorization for $\comp\vf$ or $\grave\vf$.
To see what we mean by this, recall that one point of considering $\vf'$ is given by the fact that many homological properties of
$\vf$ can be detected by $\vf'$. 
For instance, the map $\vf$ is quasi-Gorenstein if and only if $\vf'$ is quasi-Gorenstein.
We have seen similarly that many homological properties of
$\vf$ can be detected by $\psi$: e.g., under certain hypotheses, $\vf$ is quasi-Gorenstein if and only if $\psi$ is quasi-Gorenstein; see~\cite[Theorem B]{nasseh:cfwfa}.
\end{rmk}

\begin{lem}\label{lem120725b}
Every cows diagram~\eqref{eq120712a} gives rise to a commutative diagram
\begin{equation}
\label{eq120727a}
\begin{split}
\xymatrix{
R\ar[r]^-f\ar[d]_{\vf}
&\comp R \ar[r]^-{\alpha'}\ar[d]_{\comp \vf}
&S \ar[d]^{\psi}
\\
R\ar[r]^-f
&\comp R \ar[r]^-{\alpha'}
&S 
}
\end{split}
\end{equation}
of local ring homomorphisms such that the second square is cows and $\alpha=\alpha'f$ where $f\colon R\to\comp R$ is 
the natural map.
Conversely, given a cows diagram for $\comp\vf$ as in the second square of~\eqref{eq120727a},
the following diagram is cows:
\begin{equation}
\label{eq120727b}
\begin{split}
\xymatrix{
R\ar[r]^-{\alpha'f}\ar[d]_{\vf}
&S \ar[d]^{\psi}
\\
R\ar[r]^-{\alpha'f}
&S. 
}
\end{split}
\end{equation}
\end{lem}

\begin{proof}
Given a commutative diagram~\eqref{eq120712a},
since $S$ is complete the local homomorphism $\alpha$ factors through $\comp R$, so 
there is a local homomorphism $\alpha'$  making~\eqref{eq120727a} commute.
Conversely, given a commutative diagram as in the second square of~\eqref{eq120727a}, 
since the first square of~\eqref{eq120727a} commutes, it follows that
the diagram~\eqref{eq120727b} also commutes. Thus, it remains to show that
the second square of~\eqref{eq120727a} is cows if and only if~\eqref{eq120727b}
is cows. 

By construction the induced maps $R/\m\to S/\n$ and $\comp R/\m\comp R\to S/\n$ are the
same, so one is separable if and only the other is separable. Thus, it remains to show
that $\alpha'$ is weakly regular if and only if $\alpha'f$ is weakly regular.
Since $f$ is weakly regular and the composition of weakly regular maps is weakly regular,
one implication is routine. For the converse, assume that $\alpha'f$ is weakly regular.
Since $\alpha'$ and $\alpha'f$ have the same closed fibres, it suffices to show that
$\alpha'$ is flat. This follows from the sequence
$\Tor{i}{R/\m}{S}\cong\Tor[\comp R]{i}{\comp R/\m\comp R}{S}=0$
for $i\geq 1$; see~\cite[Lemme II.57]{andre:hac}
or~\cite[5.5 Proposition (F)]{avramov:hdouc}.
\end{proof}

\begin{thm}\label{thm120712a}
Let $\vf\colon R\to R$ be a 
contracting endomorphism, and let $C$ be a semidualizing $R$-complex.
Then the following conditions are equivalent:
\begin{enumerate}[\rm(i)]
\item\label{thm120712a1}
$C$ is a dualizing $R$-complex.
\item[\rm(i')]\label{prop120830a1'}
There is a cows diagram~\eqref{eq120712a} such that
$\Lotimes{S}C$ is dualizing for $S$.
\item\label{thm120712a2}
$\Lotimes{\comp R}C\sim (\Lotimes{\comp R}C)(\comp\vf^n)$ for some $n>0$.
\item[\rm(ii')]\label{thm120712a2'}
There is a cows diagram~\eqref{eq120712a} such that
$\Lotimes{S}C\sim (\Lotimes{S}C)(\psi^n)$ for some $n>0$.
\item\label{thm120712a3}
$\gcdim\vf^n<\infty$ and $\Lotimes{\comp R}C$ is derived  $(\Lotimes{\comp R}C)(\comp\vf^n)$-reflexive for some $n>0$.
\item[\rm(iii')]\label{thm120712a3'}
There is a cows diagram~\eqref{eq120712a} such that
$\gkdim{\Lotimes{S}C}\psi^n<\infty$ and
$\Lotimes{S}C$ is derived  $(\Lotimes{S}C)(\psi^n)$-reflexive for some $n>0$.
\item\label{thm120712a4}
$\gcdim\vf^n<\infty$ for infinitely many $n>0$.
\item[\rm(iv')]\label{prop120830a4'}
There is a cows diagram~\eqref{eq120712a} such that
$\gkdim{\Lotimes{S}C}\psi^n<\infty$ for infinitely many $n>0$.
\item[\rm(v')]\label{prop120830a5'} There is a cows diagram~\eqref{eq120712a} such that
$\gkdim{\Lotimes{S}C} \psi^n<\infty$  and such that $^n\!S\otimes^{\textbf{L}}_S \Rhom[S]{\Lotimes SC}{D^S}$
is derived  $\Rhom[S]{\Lotimes SC}{D^S}$-reflexive for some $n>0$, where $D^S$ is a dualizing $S$-complex.
\end{enumerate}
If $R$ has a dualizing complex $D$, then these conditions are equivalent to the following:
\begin{enumerate}[\rm(v)]
\item\label{thm120712a5} $\gcdim \vf^n<\infty$  and $^n\!R\otimes^{\textbf{L}}_R C^{\dagger_D}$
is derived  $C^{\dagger_D}$-reflexive for some $n>0$.
\end{enumerate}
\end{thm}

\begin{proof}
The equivalences (i)$\iff$(i') and (iv)$\iff$(iv') are from Fact~\ref{semi-dual complex base change} and
Lemma~\ref{lem120517a}\eqref{lem120517a1}.

For the rest of the proof, we consider two cases.

Case 1: $R$ is complete.
In this case, Theorem~\ref{thm120615a} shows that we need only prove the equivalences
(ii)$\iff$(ii'), (iii)$\iff$(iii'), and (v)$\iff$(v').
Consider a cows diagram~\eqref{eq120712a}.
Remark~\ref{disc120615a}(1) provides a a factorized pushout diagram~\eqref{eq120823a}
such that $\wti R=R$ and $\wti S=S$ and $\wti \alpha=\alpha$.
The equivalence (ii)$\iff$(ii') now  follows from Lemma~\ref{lem120615az},
and (iii)$\iff$(iii') follows from Lemmas~\ref{lem120517a}\eqref{lem120517a1}
and~\ref{lem120725a}.

For the equivalence (v)$\iff$(v') in this case,
since $R$ is complete, it has a dualizing complex $D$.
Using Lemma~\ref{lem120517a}\eqref{lem120517a1} again, we see that $\gkdim{\Lotimes{S}C} \psi^n<\infty$ if and only if $\gcdim \vf^n<\infty$.
Assume for the remainder of this paragraph that $\gkdim{\Lotimes{S}C} \psi^n<\infty$.
Since $\alpha$ is flat, there are isomorphisms in $\catd(S)$
\begin{align*}
\Rhom[S]{\Lotimes SC}{\Lotimes SD}&\simeq \Lotimes{S}{C^{\dagger_D}}
\\
^n\!S\otimes^{\textbf{L}}_S \Rhom[S]{\Lotimes SC}{\Lotimes SD}&\simeq\Lotimes{S}{(^n\!R\otimes^{\textbf{L}}_R C^{\dagger_D})}.
\end{align*}
Thus,  $^n\!S\otimes^{\textbf{L}}_S \Rhom[S]{\Lotimes SC}{\Lotimes SD}$ is derived $\Rhom[S]{\Lotimes SC}{\Lotimes SD}$-reflexive
if and only if 
$^n\!R\otimes^{\textbf{L}}_R C^{\dagger_D}$ is derived $C^{\dagger_D}$-reflexive, by~\cite[(5.10) Theorem]{christensen:scatac}.

Case 2: the general case.
Fact~\ref{semi-dual complex base change} shows that (i) is equivalent to
\begin{enumerate}[\rm(1)]
\item
$\Lotimes{\comp R}C$ is a dualizing $\comp R$-complex.
\end{enumerate}

From Fact~\ref{1} we see that conditions (iii) and (iv) are equivalent (respectively) 
to the following:

\begin{enumerate}[\rm(3)]
\item
$\gkdim{\Lotimes{\comp R}C}\comp\vf^n<\infty$ and $\Lotimes{\comp R}C$ is derived  $(\Lotimes{\comp R}C)(\comp\vf^n)$-reflexive for some $n>0$.
\item[(4)]
$\gkdim{\Lotimes{\comp R}C}\comp\vf^n<\infty$ for infinitely many $n>0$.
\end{enumerate}

Claim: Condition (ii') is equivalent to the following:

\begin{enumerate}[\rm(2')]
\item
There is a cows diagram
\begin{equation}
\label{eq120725a}
\begin{split}
\xymatrix{
\comp R\ar[r]^{\beta}\ar[d]_{\comp\vf}&S\ar[d]^{\psi}\\
\comp R\ar[r]^{\beta}&S
}
\end{split}
\end{equation}
such that
$\Lotimes[\comp R]{S}{(\Lotimes{\comp R}C)}\sim (\Lotimes[\comp R]{S}{(\Lotimes{\comp R}C)})(\psi^n)$ for some $n>0$.
\end{enumerate}

In light of Lemma~\ref{lem120725b}, this follows from the isomorphisms
\begin{align*}
(\Lotimes[\comp R]{S}{(\Lotimes{\comp R}C)})(\psi^n)&\simeq(\Lotimes{S}C)(\psi^n)
&&&\Lotimes[\comp R]{S}{(\Lotimes{\comp R}C)}&\simeq\Lotimes{S}C.
\end{align*}

Similar reasoning shows that conditions (iii') and (v') are equivalent (respectively) to the following:

\begin{enumerate}[\rm(3')]
\item
There is a cows diagram~\eqref{eq120725a} such that
$\gkdim{\Lotimes{S}C}\psi^n<\infty$ and 
$\Lotimes[\comp R]{S}{(\Lotimes{\comp R}C)}$
is derived $(\Lotimes[\comp R]{S}{(\Lotimes{\comp R}C)})(\psi^n)$-reflexive for some $n>0$.
\item[(5')]
There is a cows diagram~\eqref{eq120725a} such that
$\gkdim{\Lotimes{S}C} \psi^n<\infty$  and such that $^n\!S\otimes^{\textbf{L}}_S \Rhom[S]{\Lotimes[\comp R]{S}{(\Lotimes{\comp R}C)}}{D^S}$
is derived  $\Rhom[S]{\Lotimes[\comp R]{S}{(\Lotimes{\comp R}C)}}{D^S}$-reflexive for some $n>0$, where $D^S$ is a dualizing $S$-complex.
\end{enumerate}

Claim: if $R$ has a dualizing complex, then  condition (v) is equivalent to:
\begin{enumerate}[\rm(5)]
\item
$\gkdim{\Lotimes{\comp R}{C}} \comp\vf^n<\infty$  and the complex $\Lotimes[\comp R]{^n\!\comp R}{\Rhom[\comp R]{\Lotimes{\comp R} C}{D^{\comp R}}}$
is derived  $\Rhom[\comp R]{\Lotimes{\comp R} C}{D^{\comp R}}$-reflexive for some $n>0$
where $D^{\comp R}$ is dualizing for $\comp R$.
\end{enumerate}
Fact~\ref{1} implies that $\gkdim{\Lotimes{\comp R}{C}} \comp\vf^n<\infty$
if and only if $\gcdim\vf^n<\infty$.
From Fact~\ref{semi-dual complex base change} we know that
$\Lotimes{\comp R}{D}$ is dualizing for $\comp R$,
so we have $\Lotimes{\comp R}{D}\sim D^{\comp R}$.
The complex $\Lotimes{\comp R}{C}$ is semidualizing for $\comp R$,
hence so is 
\begin{align*}
\Rhom[\comp R]{\Lotimes{\comp R} C}{D^{\comp R}}
&\sim\Rhom[\comp R]{\Lotimes{\comp R} C}{\Lotimes{\comp R}{D}}
\simeq\Lotimes{\comp R}{C^{\dagger_D}}.
\end{align*}
In $\catd(\comp R)$ we have
\begin{align*}
\Lotimes[\comp R]{^n\!\comp R}{\Rhom[\comp R]{\Lotimes{\comp R} C}{D^{\comp R}}}
&\sim
\Lotimes[\comp R]{{}^n\!\comp R}{(\Lotimes{\comp R}{C^{\dagger_D})}}
\sim
\Lotimes{\comp R}{(\Lotimes{^n\!R}{C^{\dagger_D}})}
\end{align*}
Thus, the $\comp R$-complex
$\Lotimes[\comp R]{^n\!\comp R}{\Rhom[\comp R]{\Lotimes{\comp R} C}{D^{\comp R}}}$
is derived  $\Rhom[\comp R]{\Lotimes{\comp R} C}{D^{\comp R}}$-reflexive
if and only if
$\Lotimes{\comp R}{(\Lotimes{^n\!R}{C^{\dagger_D}})}$
is derived  $\Lotimes{\comp R}{C^{\dagger_D}}$-reflexive;
by~\cite[(5.10) Theorem]{christensen:scatac}, this second condition occurs if and only if
$\Lotimes{^n\!R}{C^{\dagger_D}}$
is derived  $C^{\dagger_D}$-reflexive.
This completes the proof of the claim.

By Case 1, conditions (1), (2), (2'), (3), (3'), (4), (5') and (5) are equivalent.
Thus, the corresponding conditions (i), (ii), etc. are equivalent.
\end{proof}

\begin{rmk}\label{rmk120906a}
As in Remark~\ref{rmk120810a}, we note here that
in Theorem~\ref{thm120712a} (and subsequent results) we have more equivalent conditions. For instance, the given conditions
are equivalent to the following:
\begin{enumerate}[\rm(i)]
\item[($\overline{\text{ii'}}$)]
For every cows diagram~\eqref{eq120712a}, we have
$\Lotimes{S}C\sim (\Lotimes{S}C)(\psi^n)$ for all $n>0$.
\end{enumerate}
\end{rmk}

Next, we consider versions of Theorems~\ref{thm120615a} and~\ref{thm120712a}
using Bass class conditions. A tool for this is the following generalization of~\cite[Theorem A]{nasseh:cmfpd} for complexes.

\begin{lem}\label{lem120808a}
Let $R\to S$ be a local ring homomorphism, and let $M$ be a homologically finite  
$S$-complex. Assume that $\vf\colon R\to R$ is a contracting endomorphism. 
Assume that there are infinitely many $n\in\bbn$ such that
there is an integer $t_n>\sup(M)$ such that
$\Ext iM{{}^n\!R}=0$ for $t_n\leq i\leq t_n+\depth(R)$. Then $\pd_R(M)<\infty$.
\end{lem}

\begin{proof}
Set $\depth(R) = d$, and
let $F$ be a degree-wise finite $S$-free resolution of $M$.
Set $j=\sup(M)$ and $M'=\coker(\partial^F_{j+1})$.
Then the complex
$$\cdots\to F_{j+1}\to F_j\to M'\to 0$$
is a degree-wise finite $S$-free resolution of $M'$.
It follows that for $i\geq j+1$ we have
\begin{align*}
\Ext {i-j}{M'}{{}^n\!R}\cong\Ext{i}{M}{{}^n\!R}.
\end{align*}
From our Ext-vanishing assumption, there are infinitely many $n\in\bbn$ such that
there is an integer $t_n'=t_n-j>0$ such that
$\Ext {i}{M'}{{}^n\!R}=0$ for $t_n'\leq i\leq t_n'+d$. 
By the proof of~\cite[Theorem A]{nasseh:cmfpd}, 
we conclude that $\pd_R(M')<\infty$, and it follows that $\pd_R(M)<\infty$.
\end{proof}

\begin{thm}\label{thm120808a}
Let $\vf\colon R\to R$ be a 
contracting endomorphism, and let $C$ be a semidualizing $R$-complex.
Then $C\sim R$ in $\catd(R)$ if and only if
$^n\!R\in\catbc(R)$ for infinitely many $n\geq 1$.
\end{thm}

\begin{proof}
The forward implication is straightforward since $\catb_R(R)$ contains all
$R$-modules. For the converse, assume that 
$^n\!R\in\catbc(R)$ for infinitely many $n\geq 1$.
In particular, there are infinitely many $n\in\bbn$ such that
$\Rhom{C}{{}^n\!R}$ is homologically bounded.
Hence, there are infinitely many $n\in\bbn$ such that
there is an integer $t_n>\sup(C)$ such that
$\Ext iC{{}^n\!R}=0$ for $t_n\leq i\leq t_n+\depth(R)$. 
Lemma~\ref{lem120808a} implies that $\pd_R(C)<\infty$,
so $C\sim R$ by~\cite[(8.1) Theorem]{christensen:scatac}.
\end{proof}

The next three lemmas are for use in the Bass class version of Theorem~\ref{thm120712a};
see Theorem~\ref{thm120808b} below.

\begin{lem}\label{prop120906a}
Let $R\xra\vf R_1$ be a ring homomorphism, and let $C$ be a semidualizing $R$-complex.
Let $L$ and $N$  be $R_1$-complexes such that $\fd_{R_1}(L)<\infty$.
If $N\in\catbc(R)$, then $\Lotimes[R_1]LN\in\catbc(R)$; the converse holds when $L$
is a faithfully flat $R_1$-module.
\end{lem}

\begin{proof}
Since $L$ has finite flat dimension over $R_1$, tensor evaluation~\cite[4.4 Lemma]{avramov:hdouc} provides
the  isomorphism
$\Lotimes[R_1]{\Rhom CN}{L}\xra\simeq\Rhom{C}{\Lotimes[R_1]{N}{L}}$.
Thus, if $\Rhom CN$ is homologically bounded, then so is $\Rhom{C}{\Lotimes[R_1]{N}{L}}$; and the converse holds 
when $L$
is a faithfully flat $R_1$-module.

Next, consider the commutative diagram wherein the upper horizontal isomorphism is from the previous paragraph:
$$\xymatrix{
\Lotimes{C}{(\Lotimes[R_1]{\Rhom CN}{L})}\ar[r]^-{\simeq}\ar[d]_-{\simeq}
& \Lotimes{C}{\Rhom{C}{\Lotimes[R_1]{N}{L}}}\ar[d]^-{\xi^C_{\Lotimes[R_1]{N}{L}}} \\
\Lotimes[R_1]{(\Lotimes C{\Rhom CN})}{L}\ar[r]^-{\Lotimes[R_1]{\xi^C_N}{L}}
&\Lotimes[R_1]{N}{L}.
}$$
From this, we conclude that $\Lotimes[R_1]{\xi^C_N}{L}$ is an isomorphism
if and only if $\xi^C_{\Lotimes[R_1]{N}{L}}$ is an isomorphism. Thus, 
if $\xi^C_N$ is an isomorphism (hence $\Lotimes[R_1]{\xi^C_N}{L}$ is an isomorphism),
then so is $\xi^C_{\Lotimes[R_1]{N}{L}}$.
When $L$ is a faithfully flat $R_1$-module and $\xi^C_{\Lotimes[R_1]{N}{L}}$ is an isomorphism, then $\Lotimes[R_1]{\xi^C_N}{L}$ is an isomorphism,
so faithful flatness implies that $\xi^C_N$ is an isomorphism.
\end{proof}

\begin{lem}\label{lem120810a}
Let $R\xra\vf R_1\xra\alpha S$ be  ring homomorphisms, and let $C$ be a semidualizing $R$-complex.
Assume that $\alpha$ is flat. If $R_1\in\catbc(R)$, then $S\in\catbc(R)$; the converse holds when $\alpha$
is faithfully flat, e.g., when $\alpha$ is local.
\end{lem}

\begin{proof}
Use $N=R_1$ and $L=S$ in Lemma~\ref{prop120906a}.
\end{proof}

\begin{lem}\label{lem120906a}
For every cows diagram~\eqref{eq120712a} and every $n\in\bbn$, one has 
$^n\!R\in\catbc(R)$ if and only if 
$^n\!S\in\catb_{\Lotimes SC}(S)$.
\end{lem}

\begin{proof}
The cows diagram~\eqref{eq120712a} yields a commutative diagram
$$\xymatrix{
R\ar[r]^{\alpha}\ar[d]_{\vf^n}&S\ar[d]^{\psi^n}\\
R\ar[r]^{{\alpha}}&S.
}$$
Since $\alpha$ is faithfully flat, Lemma~\ref{lem120810a} shows that $^n\!R\in\catbc(R)$ if and only if 
$^n\!S\in\catbc(R)$, and~\cite[(5.3.a) Proposition]{christensen:scatac} shows that
$^n\!S\in\catbc(R)$ if and only if 
$^n\!S\in\catb_{\Lotimes SC}(S)$.
\end{proof}

\begin{thm}\label{thm120808b}
Let $\vf\colon R\to R$ be
a contracting endomorphism, and let $C$ be a semidualizing $R$-complex.
Then the following conditions are equivalent:
\begin{enumerate}[\rm(i)]
\item\label{thm120808b0}
$R$ is Gorenstein.
\item\label{thm120808b1}
$\gcdim\vf^m<\infty$ for all $m>0$, and $^n\!R\in\catbc(R)$ for all $n>0$.
\item[\rm(ii')]
For every cows diagram~\eqref{eq120712a}, one has 
$\gkdim{\Lotimes{S}C}\psi^m<\infty$ for all $m>0$, and $^n\!S\in\catb_{\Lotimes SC}(S)$ for all $n>0$.
\item\label{thm120808b2}
$\gcdim\vf^m<\infty$ for infinitely many $m>0$, and $^n\!R\in\catbc(R)$ for some $n>0$.
\item[\rm(iii')]
There is a cows diagram~\eqref{eq120712a} such that
$\gkdim{\Lotimes{S}C}\psi^m<\infty$ for infinitely many $m>0$, and $^n\!S\in\catb_{\Lotimes SC}(S)$ for some $n>0$.
\item\label{thm120808b3}
$\gcdim\vf^m<\infty$ for some $m>0$, and $^n\!R\in\catbc(R)$ for infinitely many $n>0$.
\item[\rm(iv')]
There is a cows diagram~\eqref{eq120712a} such that
$\gkdim{\Lotimes{S}C}\psi^m<\infty$ for some $m>0$, and $^n\!S\in\catb_{\Lotimes SC}(S)$ for infinitely many $n>0$.
\end{enumerate}
\end{thm}

\begin{proof}
The implications (ii)$\implies$(iii) and
(ii)$\implies$(iv) are trivial.
The equivalences (ii)$\iff$(ii'), (iii)$\iff$(iii'), and (iv)$\iff$(iv') follow from Lemma~\ref{lem120906a}.

(i)$\implies$(ii)
Assume that
$R$ is Gorenstein.
Then we know from~\cite[(8.6) Corollary]{christensen:scatac} that $C\sim R$,
so $\catbc(R)=\catb_R(R)$ contains every $R$-module, in particular
$^n\!R\in\catbc(R)$ for all $n>0$.
Also, since $R$ is Gorenstein,
we have $\gcdim\vf^m=\gdim\vf^m<\infty$ for all $m>0$
by~\cite[6.6.~Theorem]{iyengar:golh}.

(iii)$\implies$(i)
Assume that
$\gcdim\vf^m<\infty$ for infinitely many $m>0$, and $^n\!R\in\catbc(R)$ for some $n>0$.
Theorem~\ref{thm120712a} implies that $C$ is dualizing for $R$.

Case 1: $R$ is complete. In this case, $\vf^n$ has a Cohen factorization $R\xra{\dot\tau}R'\xra{\tau'}R$.
From~\cite[(5.3.b) Proposition]{christensen:scatac} the condition $^n\!R\in\catbc(R)$ implies that
$R\in\catb_{\Lotimes{R'}{C}}(R')$. As $C$ is dualizing for $R$ and $\dot\tau$ is weakly regular,
it follows that $\Lotimes{R'}{C}$ is dualizing for $R'$.
Because of~\cite[4.4.~Theorem]{christensen:ogpifd}, we conclude that
$\text{G-}\id_{R'}(R)<\infty$.\footnote{See~\cite{christensen:ogpifd,foxby:cmfgidgr} for background on G-injective dimension.}
In particular, the local ring $R'$ has a cyclic module of finite G-injective dimension,
so $R'$ is Gorenstein by~\cite[Theorem A]{foxby:cmfgidgr}. The fact that $\dot\tau$ is flat and local
implies that $R$ is Gorenstein.

Case 2: the general case. The ring $R$ is Gorenstein if and only if  $\comp R$ is Gorenstein. 
Since $\gcdim\vf^m<\infty$ for infinitely many $m>0$,
Fact~\ref{1} implies that $\gkdim{\Lotimes{\comp R}{C}}\comp\vf^m<\infty$ for infinitely many $m>0$.
By Case 1, it suffices to show that the assumption $^n\!R\in\catbc(R)$ implies that 
$^n\!\comp R\in\catb_{\Lotimes{\comp R}{C}}(\comp R)$.
Consider the commutative diagram of local ring homomorphisms
$$\xymatrix{
R\ar[r]\ar[d]_{\vf^n}
&\comp R\ar[d]^{\comp\vf^n}\\
R\ar[r]&\comp R}$$
where the unspecified maps are the natural ones. The assumption
$^n\!R\in\catbc(R)$ implies that 
$^n\!\comp R\in\catbc(R)$ by Lemma~\ref{lem120810a}.
From~\cite[(5.3.b) Proposition]{christensen:scatac} we conclude that
$^n\!\comp R\in\catb_{\Lotimes{\comp R}{C}}(\comp R)$, as desired.

(iv)$\implies$(i)
Assume that
$\gcdim\vf^m<\infty$ for some $m>0$, and $^n\!R\in\catbc(R)$ for infinitely many $n>0$.
Theorem~\ref{thm120808a} implies that $C\sim R$ in $\catd(R)$.
Thus, the assumption $\gcdim\vf^m<\infty$ translates to $\gdim\vf^m<\infty$,
and we conclude from~\cite[6.6.~Theorem]{iyengar:golh} that $R$ is Gorenstein.
\end{proof}

\section{Results Specific to the Frobenius Endomorphism}\label{sec120515a}

We begin this section with a combination of~\cite[Proposition 12.2.7]{avramov:holh} and~\cite[Proposition (0.10.3.1)]{grothendieck:ega3-1}.

\begin{lem}\label{Factorization of Frobenius}
Let $(R,\fm,k)$ be a  local ring of prime
characteristic $p>0$, and let $k\subseteq K$ be a field extension.
Then there is a commutative diagram
of local ring homomorphisms
$$\xymatrix{
&\comp R\ar[r]^{\alpha}\ar[ld]_{\dot{\varphi}_{\comp R}}\ar[dd]^<<<<<{f^n_{\comp R}}& S\ar[dd]_<<<<<{f^n_S}\ar[rd]^{\dot{\varphi}_S}&&\\
R'\ar[rd]_{\varphi'_{\comp R}}\ar[rrr]_-{\alpha'}&&&S'\ar[ld]^{\varphi'_S}&\\
&\comp R\ar[r]^{\alpha}&S.
}$$
such that
the maps $\alpha$ and $\alpha'$ are weakly \'etale, 
the rings $S$ and $S'$ are  complete,
the diagrams $\comp R\xra{\dot\vf_{\comp R}}R'\xra{\vf_{\comp R}'}\comp R$ and $S\xra{\dot\vf_S}S'\xra{\vf_S'}S$
are minimal Cohen factorizations of $f^n_{\comp R}$ and $f^n_S$, respectively,
the natural map $\Lotimes[R']{S'}{\comp R}\to S$ is an isomorphism, and the induced map $R/\m\to S/\m S$ is the given field extension $k\subseteq K$.
(In particular, this is a factorized pushout diagram.)
If $K$ is $F$-finite, then so are $S$ and $S'$.
\end{lem}

\begin{proof}
\newcommand{\chr}{\operatorname{char}}
\newcommand{\ba}{\mathbf{a}}
This conclusion is unchanged if we replace $R$ by $\comp R$, so we assume  that $R$ is complete. 
By Cohen's Structure Theorem there exist integers $e,m\geq 0$ and elements
$f_1,\cdots,f_m\in k[\![x_1,\ldots,x_e]\!]$ such that
$R\cong k[\![x_1,\ldots,x_e]\!]/(f_1,\ldots,f_m)$, and
 the images of $x_1,\ldots,x_e$ in $R$ minimally generate $\m$. 
Set $\x=x_1,\ldots,x_e$, and use the notation
$\x^{\ba}=x_1^{a_1}\cdots x_e^{a_e}$ for all $\ba=(a_1,\ldots,a_e)\in\bbn^e$.
Also, set $\f=f_1,\cdots,f_m$.
We identify $R$ with the ring $k[\![\x]\!]/(\f)$ for the remainder of the proof.

Let $S=K[\![\x]\!]/(\f)$, and let $\alpha\colon R\to S$ be induced by the inclusion $k\subseteq K$. Then $S$ is a
complete local ring, and $\alpha$ makes $S$
into a  local $R$-algebra of characteristic $p$.
It is straightforward to show that the map
$k[\![\x]\!]\to K[\![\x]\!]$ induced by the inclusion
$k\subseteq K$ is flat (e.g., using~\cite[Exercise 22.3]{matsumura:crt}). Hence, $\alpha$ is flat by base-change.
Moreover, $\alpha$ is weakly \'etale
since the maximal ideal of $S$ is $(\x)S=\m S$ by construction.

We use the following notation of~\cite[Proposition 12.2.7]{avramov:holh}.
For  $g=\sum_{\ba\in\bbn^e}r_{\ba}\x^{\ba}$, set
$g^{[p^n]}=\sum_{\ba\in\bbn^e}r_{\ba}^{p^n}\x^{\ba}$.
Then a minimal Cohen factorization of $f_R^n$ is given by the following maps.
The weakly regular part is 
$\dot\vf_R\colon k[\![\x]\!]/(\f)\to k[\![\x]\!]/(\f^{[p^n]})[\![\y]\!]$
given by $\dot\vf_R(g+(\f)):=g^{[p^n]}+(\f^{[p^n]})$ where $\y=y_1,\ldots,y_e$ is another
list of variables.
The surjective part is the composition $\vf_R'=\phi''_R\rho_R$ where $\phi''_R$
and $\rho_R$ are defined next.
First, we have $\rho_R\colon k[\![\x]\!]/(\f^{[p^n]})[\![\y]\!]\to k[\![\x]\!]/(\f^{p^n})$
which leaves the elements of $k$ fixed and such that
$\rho_R(x_i+(\f^{[p^n]}))=x_i^{p^n}+(\f^{p^n})$
and 
$\rho_R(y_i+(\f^{[p^n]}))=x_i+(\f^{p^n})$.
Next, we have the natural surjection $\phi''_R\colon k[\![\x]\!]/(\f^{p^n})\to k[\![\x]\!]/(\f)$.
As is observed in the proof of~\cite[Proposition 12.2.7]{avramov:holh}, $\Ker(\rho_R)$
is generated by the sequence $x_1-y_1^{p^n},\ldots,x_e-y_e^{p^n}$.
The factorization of $f^n_S$ is defined similarly in terms of
$\dot\vf_S$, $\rho_S$, and $\phi''_S$.

This provides the outer edges of the following commutative diagram:
$$
\xymatrix@C=4mm{
&\!\!\!\!R=k[\![\x]\!]/(\f)\ar[r]^-{\alpha}\ar[ld]_{\dot\vf_R}\ar[ddd]^<<<<<{f^n_R}&
S=K[\![\x]\!]/(\f)\ar[ddd]^<<<<<{f^n_S}\ar[rd]^{\dot\vf_S}\\
R':=k[\![\x]\!]/(\f^{[p^n]})[\![\y]\!]\ar[d]_{\rho_R}\ar[rrr]^-{\alpha'}\ar[rdd]^<<<<<<<<<<<<{\vf_R'}&&&
S':=K[\![\x]\!]/(\f^{[p^n]})[\![\y]\!]\ar[d]^{\rho_S}\ar[ldd]_<<<<<<<<<<<<{\vf_S'}\\
k[\![\x]\!]/(\f^{p^n})\ar[dr]_{\phi''_R}\ar[rrr]^-{\alpha^*}&&& 
K[\![\x]\!]/(\f^{p^n})\ar[ld]^{\phi''_S}
\\
&\!\!\!\!R=k[\![\x]\!]/(\f)\ar[r]^-{\alpha}&S=K[\![\x]\!]/(\f).
}
$$
The maps $\alpha'$ and $\alpha^*$ are induced by the inclusion $k\subseteq K$.
Hence, they are weakly \'etale, by the same proof as
for $\alpha$.

Next, we verify the pushout condition.
Since $\phi''_R\rho_R$ and $\phi''_S\rho_S$ are surjective and $\alpha'$ is flat,
it suffices to show that $\Ker(\phi''_S\rho_S)$ is generated by the images in $K[\![\x]\!]/(\f^{[p^n]})[\![\y]\!]$
of the generators of $\Ker(\phi''_R\rho_R)$. From the above discussion, we know that
$\Ker(\rho_R)$
is generated by the sequence $x_1-y_1^{p^n},\ldots,x_e-y_e^{p^n}$.
Since $\rho_S$ is constructed exactly like $\rho_R$, we conclude that $\Ker(\rho_S)$
is generated by the sequence $x_1-y_1^{p^n},\ldots,x_e-y_e^{p^n}$.
Also, by construction, $\Ker(\phi''_R)$ and $\Ker(\phi''_S)$
are generated by the images of $\f$ in the respective rings.
Thus, $\Ker(\phi''_R\rho_R)$ and $\Ker(\phi''_S\rho_S)$ are both generated by
$x_1-y_1^{p^n},\ldots,x_e-y_e^{p^n},\f$
as desired.

Finally, if $K$ is $F$-finite, then $S$ and $S'$
are $F$-finite by~\cite[Corollary 2.6]{kunz:onrcp}. 
\end{proof}

The next result is like Theorem~\ref{thm120712a}.

\begin{thm}\label{thm120712a'}
Let $(R,\fm,k)$ be a  local ring of prime
characteristic $p>0$, and let $C$ be a semidualizing $R$-complex.
Then the following are equivalent:
\begin{enumerate}[\rm(i)]
\item\label{thm120712a'1}
$C$ is a dualizing $R$-complex.
\item[\rm(i')]\label{prop120910a1'}
There is a complete weakly \'etale $F$-finite local $R$-algebra $S$ such that
$\Lotimes{S}C$ is dualizing for $S$.
\item\label{thm120712a'2}
$\Lotimes{\comp R}C\sim (\Lotimes{\comp R}C)(\comp f_R^n)$ for some $n>0$.
\item[\rm(ii')]\label{thm120712a'2'}
There is a complete weakly \'etale $F$-finite local $R$-algebra $S$  such that
$\Lotimes{S}C\sim (\Lotimes{S}C)(f_S^n)$ for some $n>0$.

\item\label{thm120712a'3}
$\gcdim f_R^n<\infty$ and $\Lotimes{\comp R}C$ is derived  $(\Lotimes{\comp R}C)(f_{\comp R}^n)$-reflexive for some $n>0$.
\item[\rm(iii')]\label{thm120712a'3'}
There is a  complete weakly \'etale $F$-finite local $R$-algebra $S$  such that
for some $n>0$ we have $\gkdim{\Lotimes{S}C} f_S^n<\infty$ and 
$\Lotimes{S}C$ is derived  $(\Lotimes{S}C)(f_S^n)$-reflexive.
\item\label{thm120712a'4}
$\gcdim f_R^n<\infty$ for infinitely many $n>0$.
\item[\rm(iv')]\label{prop120910a4'}
There is a complete weakly \'etale $F$-finite local $R$-algebra $S$ such that  for infinitely many $n>0$ we have
$\gkdim{\Lotimes{S}C} f_S^n<\infty$.
\item[\rm(v')]\label{prop120910a5'} 
There is a complete weakly \'etale $F$-finite local $R$-algebra $S$  such that for some $n>0$ we have
$\gkdim{\Lotimes{S}C} f_S^n<\infty$  and $^n\!S\otimes^{\textbf{L}}_S \Rhom[S]{\Lotimes SC}{D^S}$
is derived  $\Rhom[S]{\Lotimes SC}{D^S}$-reflexive, where $D^S$ is a dualizing $S$-complex.
\end{enumerate}
If $R$ has a dualizing complex $D$, then these conditions are equivalent to the following:
\begin{enumerate}[\rm(v)]
\item\label{thm120712a'5} $\gcdim f_R^n<\infty$  and $^n\!R\otimes^{\textbf{L}}_R C^{\dagger_D}$
is derived  $C^{\dagger_D}$-reflexive for some $n>0$.
\end{enumerate}
\end{thm}

\begin{proof}
The proof is  like that of Theorem~\ref{thm120712a}.
The only difference is in the equivalences (ii)$\iff$(ii') and (iii)$\iff$(iii'), in which we use
Lemma~\ref{Factorization of Frobenius}, where $K$ is an algebraic closure of $k$.
\end{proof}

The last result of this section is proved like Theorem~\ref{thm120808b}.

\begin{thm}\label{thm120910a}
Assume that $R$ is a local ring of prime characteristic $p>0$, and let $C$ be a semidualizing $R$-complex.
Then the following conditions are equivalent:
\begin{enumerate}[\rm(i)]
\item\label{thm120910a0}
$R$ is Gorenstein.
\item\label{thm120910a1}
$\gcdim f_R^m<\infty$ for all $m>0$, and $^n\!R\in\catbc(R)$ for all $n>0$.
\item[\rm(ii')]
For every complete weakly \'etale $F$-finite local $R$-algebra $S$, one has $^n\!S\in\catb_{\Lotimes SC}(S)$ for all $n>0$, and 
$\gkdim{\Lotimes{S}C} f_S^m<\infty$ for all $m>0$.
\item\label{thm120910a2}
$\gcdim f_R^m<\infty$ for infinitely many $m>0$, and $^n\!R\in\catbc(R)$ for some $n>0$.
\item[\rm(iii')]
There is a complete weakly \'etale $F$-finite local $R$-algebra $S$  such that $^n\!S\in\catb_{\Lotimes SC}(S)$ for some $n>0$, and 
$\gkdim{\Lotimes{S}C} f_S^m<\infty$ for infinitely many $m>0$.
\item\label{thm120910a3}
$\gcdim f_R^m<\infty$ for some $m>0$, and $^n\!R\in\catbc(R)$ for infinitely many $n>0$.
\item[\rm(iv')]
There is a complete weakly \'etale $F$-finite local $R$-algebra $S$  such that $^n\!S\in\catb_{\Lotimes SC}(S)$ for infinitely many $n>0$, and 
$\gkdim{\Lotimes{S}C} f_S^m<\infty$ for some $m>0$.
\end{enumerate}
\end{thm}

\appendix
\section{A Construction of Endomorphisms}\label{sec120627a}

The point of this section is found in Proposition~\ref{prop120629a}, which guarantees the existence of 
non-trivial cows diagrams~\eqref{eq120712a}; see Remark~\ref{rmk120803a}.

\begin{lem}\label{lem120627a}
Let $\alpha\colon k\to k$ be a field endomorphism.
Then there is a   commutative diagram of field extensions
\begin{equation}\label{eq120627a}
\begin{split}
\xymatrix{
k\ar[r]^{\beta}\ar[d]_{\alpha}
&K\ar[d]^{\comp\alpha} \\
k\ar[r]^{\beta}
&K
}
\end{split}
\end{equation}
such that $\comp \alpha$ is an isomorphism.
Moreover, if $\alpha$ is separable, then so is $\beta$.
\end{lem}

\begin{proof}
Let $K$ be the direct limit in the category of fields and field extensions
of the directed system
$k\xra{\alpha}k\xra{\alpha}k\xra{\alpha}\cdots$.
The universal mapping property for direct limits provides a unique field endomorphism
$\comp\alpha\colon K\to K$ making the following diagram commute
$$\xymatrix{
k\ar[r]^{\alpha}
\ar[dd]_{\alpha}
\ar[rrrd]_<<<<<<{\beta_1}
&k\ar[r]^{\alpha}
\ar[dd]_{\alpha}
\ar[rrd]^{\beta_2}
&\cdots
\\
&&&K\ar@{-->}[dd]^{\comp\alpha}
\\
k\ar[r]^{\alpha}
\ar[rrrd]_<<<<<<{\beta_1}
&k\ar[r]^{\alpha}
\ar[rrd]^{\beta_2}
&\cdots
\\
&&&K
}$$
where the maps $\beta_i$ are the universal ones for $K$. 
With $\beta:=\beta_1$, we have the commutativity of~\eqref{eq120627a},
so we need to show that $\comp\alpha$ is an isomorphism.
Since $\comp\alpha$ is a morphism of fields, it is injective, so we need only verify surjectivity.
For this, let $x\in K=\cup_{i=1}^\infty\im(\beta_i)$. Then there is an index $i$ 
and an element $y\in k$ such that
$$x=\beta_i(y)=\beta_{i+1}(\alpha(y))=\comp\alpha(\beta_{i+1}(y))\in\im(\comp\alpha).
$$
this yields the surjectivity of $\comp\alpha$.

To complete the proof, assume that $\alpha$ is separable. 
It follows that $\alpha^n$ is separable for each $n\geq 1$.
To show that $\beta$ is separable, we need to show that for every intermediate field
$k\to F\to K$ such that $F$ is finitely generated as a field extension of $k$, the extension $k\to F$ is separably generated;
see~\cite[Theorem VI.2.10]{hungerford:a}.
Again, $K$ is the union of the images
of the $\beta_i$.
The finite generation condition on $F$ implies that the generators for
$F$ over $k$ lie in some ``finite stage'' of the limit. In other words, the commutative diagram
$$\xymatrix{
k\ar[r]\ar[rd]_{\beta}&F\ar[d]\\ & K}$$
gives rise to another commutative diagram
$$\xymatrix{
k\ar[r]\ar[rd]_{\alpha^n}&F\ar[d]\\ & k}$$
for some $n\geq 1$. Since $\alpha^n$ is separable, the intermediate
extension $k\to F$ is separably generated, as desired.
\end{proof}

The next result takes a construction of Grothendieck~\cite[Proposition (0.10.3.1)]{grothendieck:ega3-1} and manipulates it a bit,
similarly to Lemma~\ref{Factorization of Frobenius}.

\begin{lem}\label{lem120629b}
Let $\vf\colon (R,\m,k)\to(R,\m,k)$ be a contracting endomorphism,
and consider a commutative diagram of field extensions
\begin{equation}\label{eq120629a}
\begin{split}
\xymatrix{
k\ar[r]^{\beta}\ar[d]_{\ol\vf}
&K\ar[d]^{\comp\alpha} \\
k\ar[r]^{\beta}
&K
}
\end{split}
\end{equation}
where $\ol\vf$ is the map induced by $\vf$.
Assume that $\beta$ is separable.
Then there is a commutative diagram of local ring homomorphisms
\begin{equation}\label{eq120629b}
\begin{split}
\xymatrix{
(R,\m,k)\ar[r]^{\wti\beta}\ar[d]_{\vf}
&(S,\m S,K)\ar[d]^{\psi} \\
(R,\m,k)\ar[r]^{\wti\beta}
&(S,\m S,K)
}
\end{split}
\end{equation}
such that $S$ is complete,
$\wti\beta$ is weakly \'etale,
$\psi$ is a contracting endomorphism, and the diagram induced by~\eqref{eq120629b}
on residue fields is~\eqref{eq120629a}.
\end{lem}

\begin{proof}
By~\cite[Proposition (0.10.3.1)]{grothendieck:ega3-1} there is a weakly \'etale  local
ring homomorphism $\wti\beta\colon(R,\m,k)\to(S,\m S,K)$.
Replace $S$ with its completion if necessary to assume that $S$ is complete.
Since the induced map $\beta\colon k\to K$ is separable, we conclude that
$\wti\beta$ is formally smooth. Since $S$ is complete, a standard application of smoothness
provides a local ring homomorphism $\psi\colon S\to S$ such that 
(1) $\psi$ induces $\comp\alpha$ on residue fields, and (2) $\psi$ respects the $R$-algebra
structures given by $\wti\beta$ and $\wti\beta\vf$, that is, such that the
diagram~\ref{eq120629b} commutes.

It remains to show that $\psi$ is contracting.
For this, fix an integer $i\geq 1$ such that $\vf^i(\m)\subseteq\m^2$. Then
$$\psi^i(\m S)\subseteq \psi^i(\m S)S=\psi^i(\wti\beta(\m))S=\wti\beta(\vf^i(\m))S\subseteq\wti\beta(\m^2)S=(\m S)^2$$
as desired.
\end{proof}

\begin{prop}\label{prop120629a}
Let $\vf\colon (R,\m,k)\to(R,\m,k)$ be a contracting endomorphism
such that $R/\vf(\m)R$ is artinian and the induced map $\ol\vf\colon k\to k$ is separable.
Then there is a commutative diagram of local ring homomorphisms
\begin{equation}\label{eq120629c}
\begin{split}
\xymatrix{
(R,\m,k)\ar[r]^{\wti\beta}\ar[d]_{\vf}
&(S,\m S,K)\ar[d]^{\psi} \\
(R,\m,k)\ar[r]^{\wti\beta}
&(S,\m S,K)
}
\end{split}
\end{equation}
such that $\wti\beta$ is weakly \'etale and the induced map $k\to K$ is separable (hence $\wti\beta$ is formally smooth),
$S$ is complete, and
$\psi$ is a module-finite contracting endomorphism.
\end{prop}

\begin{proof}
Lemma~\ref{lem120627a}
provides a   commutative diagram of field extensions
\begin{equation}\label{eq120710a}
\begin{split}
\xymatrix{
k\ar[r]^{\beta}\ar[d]_{\alpha}
&K\ar[d]^{\comp\alpha} \\
k\ar[r]^{\beta}
&K
}
\end{split}
\end{equation}
such that $\comp \alpha$ is an isomorphism and $\beta$ is separable.
Hence, Lemma~\ref{lem120629b} implies that
we have a commutative diagram of local ring homomorphisms
\begin{equation}\label{eq120710b}
\begin{split}
\xymatrix{
(R,\m,k)\ar[r]^{\wti\beta}\ar[d]_{\vf}
&(S,\m S,K)\ar[d]^{\psi} \\
(R,\m,k)\ar[r]^{\wti\beta}
&(S,\m S,K)
}
\end{split}
\end{equation}
such that $\wti\beta$ is weakly \'etale,
$\psi$ is a contracting endomorphism, and the diagram induced by~\eqref{eq120710b}
on residue fields is~\eqref{eq120710a}.
To avoid ambiguous notation, given an $S$-module $M$, 
let $^{\psi}\!M$ denote the $S$-module structure on $M$ given by restriction of scalars along $\psi$.

It remains to show that $\psi$ is module-finite. 
Since $S$ is complete and local,
the complete Nakayama's Lemma~\cite[Theorem 8.4]{matsumura:crt}
says that it suffices to show that $^\psi\!(S/\psi(\m S)S)$ has finite length over $S$.
Since the  map $\comp\alpha\colon S/\m S\to S/\m S$ induced by $\psi$ is an isomorphism,
if $M$ is an $S$-module of finite length, then $^\psi \! M$ also has finite length;
in fact, the lengths are the same in this case.
Thus, it remains to show that $S/\psi(\m S)S$  has finite length over $S$.

By assumption, the quotient ring $R/\vf(\m)R$ is artinian.
This means that $\m^n\subseteq\vf(\m)R$ for $n\gg 0$, so 
$$(\m S)^n=\wti\beta(\m^n)S\subseteq\wti\beta(\vf(\m))S=\psi(\wti\beta(\m))S
=\psi(\m S)S.$$
We conclude that $S/\psi(\m S)S$ is artinian,
so it has finite length over $S$.
\end{proof}

\section*{Acknowledgments} 
We are grateful to Hamid Rahmati for pointing out~\cite[Remark 5.9]{avramov:himce}.

%\bibliography{../+new}

\providecommand{\bysame}{\leavevmode\hbox to3em{\hrulefill}\thinspace}
\providecommand{\MR}{\relax\ifhmode\unskip\space\fi MR }
\providecommand{\MRhref}[2]{%
  \href{http://www.ams.org/mathscinet-getitem?mr=#1}{#2}
}
\providecommand{\href}[2]{#2}

\end{document}